\documentclass[12pt]{amsart}

\usepackage{amsmath}
\usepackage{amsthm}
\usepackage{hyperref}
\usepackage[margin=1.0in]{geometry}

\numberwithin{equation}{section}

\newtheorem{prop}{Proposition}
\newtheorem{lemma}[prop]{Lemma}

\newtheorem{thm}[prop]{Theorem}
\newtheorem{cor}[prop]{Corollary}

\numberwithin{prop}{section}

\theoremstyle{definition}
\newtheorem{defn}[prop]{Definition}
\newtheorem{ex}[prop]{Example}
\newtheorem{rmk}[prop]{Remark}

\newcommand{\brs}[1]{\left| #1 \right|}

\newcommand{\gD}{\Delta}

\newcommand{\gs}{\sigma}

\newcommand{\gl}{\lambda}

\newcommand{\gw}{\omega}
\newcommand{\ga}{\alpha}
\newcommand{\gb}{\beta}

\newcommand{\N}{\nabla}
\newcommand{\FF}{\mathcal F}

\renewcommand{\bar}[1]{\overline{#1}}

\renewcommand{\i}{\sqrt{-1}}

\newcommand{\IP}[1]{\left<#1\right>}

\DeclareMathOperator{\Rc}{Rc}

\DeclareMathOperator{\tr}{tr}

\DeclareMathOperator{\Id}{Id}
\DeclareMathOperator{\divg}{div}

\DeclareMathOperator{\spn}{span}

\DeclareMathOperator{\Spin}{Spin}

\begin{document}

\title[The canonical symmetry reduction of string backgrounds]{The canonical symmetry reduction of string backgrounds}

\date{\today}

\begin{abstract} String backgrounds, defined here as metric connections with skew-symmetric torsion and reduced holonomy, yield generalized Ricci solitons relative to the Lee vector field.  By a variational argument using the string action, they are also gradient generalized Ricci solitons relative to a potential function.  These two observations combine to yield a canonical symmetry, and in this work we derive fundamental features of the transverse geometry, and rigidity phenomena.  We prove in a unified conceptual fashion that the transverse geometry satisfies the string generalized Ricci soliton equations (a simplified Hull-Strominger system) in many settings including almost Hermitian, almost contact, $SU(3)$, $G_2$, and $\Spin(7)$ geometry.  We also show that the transverse geometry is always conformally co-closed, with the conformal factor given by the associated soliton potential.
\end{abstract}

\author{Aaron Kennon}
\address{Rowland Hall\\
         University of California, Irvine\\
         Irvine, CA 92617}
\email{\href{mailto:kennona@uci.edu}{kennona@uci.edu}}

\author{Jeffrey Streets}
\address{Rowland Hall\\
         University of California, Irvine\\
         Irvine, CA 92617}
\email{\href{mailto:jstreets@uci.edu}{jstreets@uci.edu}}

\thanks{The authors are supported by the NSF via DMS-2342135.  They would like to thank Beatrice Brienza and Stefan Ivanov for pointing out errors in an earlier draft.}

\maketitle


%
%
%
%
\section{Introduction}

String backgrounds arise via metric connections $\N$ with skew-symmetric torsion admitting a parallel spinor and reduced holonomy, yielding stringy generalizations of the classical special holonomy manifolds i.e. Calabi-Yau, $G_2$, and $\Spin(7)$ manifolds.  In physics literature, constraints motivated by the standard model are placed on the torsion, leading to different flavors of the physical Hull-Strominger system, a subject of intense interest in recent years (cf. for instance \cite{becker2003compactifications, becker2004compactifications, cardoso2003non,garcia2024pluriclosed, gauntlett2004superstrings, gutowski2002deformations, Ivanovstring, ivanov2001no,phong2019geometric,picard2024strominger,StromingerSST} for the case of $SU(n)$ holonomy on $2n$-manifolds, \cite{lotay2024coupled,de2025mathrm,de2017restrictions, de2018infinitesimal,friedrich2002parallel, friedrich2003killing, galdeano2024heterotic,gauntlett2001fivebranes, ivanov2023riemannianG2} for $G_2$ holonomy on $7$-manifolds, and \cite{gauntlett2004g, gauntlett2001fivebranes,ivanov2023riemannianspin7} for $\Spin(7)$ holonomy on $8$-manifolds).  In this work we will focus on the geometrically motivated case when the torsion is simply closed, although this is fundamentally tied to a simplified version of the Hull-Strominger system as explained below.

On a Hermitian manifold a unique connection with skew-symmetric torsion always exists, introduced by Bismut \cite{Bismut} in his work on the local index theorem on complex manifolds (cf. also \cite{StromingerSST}).  For more exotic geometries, foundational work in this direction was done by Friedrich-Ivanov \cite{friedrich2002parallel}, who identified necessary and sufficient conditions for a connection with skew-symmetric torsion to be compatible with almost contact structures, almost Hermitian, and $G_2$ structures, while the $\Spin(7)$ case was treated in \cite{ivanov2004connections}.  In each of these settings, we recover the string background equations by imposing two further conditions, namely that the torsion $H$ is \emph{closed}, and that $\N$ has \emph{reduced holonomy}.  In the setting of complex manifolds the condition that $H$ is closed is just the condition that the metric is pluriclosed, and the condition that $\N$ has holonomy in $SU(n)$ then yields a \emph{Bismut Hermitian-Einstein} (BHE) structure \cite{PCF,JFS}.  Closely related considerations yield natural notions of special geometries on almost contact and almost Hermitian manifolds, which we formalize below.  As part of a recent series of works, Ivanov-Stanchev \cite{ivanov2023riemannianG2} proved the beautiful fact that in the $G_2$ context, simply imposing that the torsion is closed already implies the relevant equations of motion.  These are called strong torsion $G_2$ structures and have been recently studied in \cite{fino2023twisted, fino2025some}.  The analogous statement was shown in the $\Spin(7)$ setting in \cite{ivanov2023riemannianspin7}.  We refer also to \cite{howe2010covariantly} for similar results in the physics literature

In all of these settings the equation of motion takes the same shape.  In particular, each of these geometric settings defines a relevant notion of \emph{Lee form} \cite{lee1943kind}, first playing a role in Hermitian geometry, and later in $G_2$ geometry (\cite{fernandez1982riemannian,bryant1987metrics}).  In all cases the Lee form, which we generically denote by $\theta$ here, measures the effect of a conformal transformation of some spinor defining the relevant geometric structure.  As elucidated in a series of classical and recent works of Ivanov and collaborators \cite{Ivanovstring, ivanov2004connections, ivanov2024riemannian, ivanov2023riemannianG2, ivanov2023riemannianspin7}, it is known that the reduced holonomy geometries described above all satisfy
\begin{align} \label{f:holonomy}
    \Rc^{\N} +\ \N \theta = 0,
\end{align}
This is the equation of a \emph{non-gradient generalized Ricci soliton} (cf. \cite{GRFbook} Ch. 4), in particular thus corresponding to a self-similar solution of generalized Ricci flow.

A subtly different physical argument has also been used to define string backgrounds, via the variational theory of the string effective action (cf. \cite{polchinski2005string}), and associated Perelman $\gl$-functional (cf. \cite{OSW}).  This functional is defined for a pair $(g, H)$ of a Riemannian metric and closed three-form in a fixed deRham class as:
\begin{align*}
    \mathcal F(g, H, f) := \int_M \left( R - \tfrac{1}{12} \brs{H}^2 + \brs{\N f}^2 \right) e^{-f} dV_g, \qquad \gl(g, H) := \inf_{\{f| \int_M e^{-f} dV_g=1\}} \mathcal F(g, H, f).
\end{align*}
Critical points of $\gl$ are \emph{gradient generalized Ricci solitons}, a notion naturally captured by the associated Bismut/Strominger connection \cite{Bismut, StromingerSST}
\begin{align*}
    \N = D + \tfrac{1}{2} g^{-1} H.
\end{align*}
In particular, if $(g, H)$ is a critical point for $\gl$ and $f$ is the unique function achieving the infimum defining $\gl$, then one has
\begin{align*}
    \Rc^{\N} + \N df = 0.
\end{align*}
Due to the monotonicity formula for $\gl$ along generalized Ricci flow, a structure satisfying (\ref{f:holonomy}) is actually a critical point for $\gl$, hence also a gradient soliton.  Combining these two equations thus yields the existence of a canonical \emph{Bismut parallel} vector field
\begin{align*}
    V = g^{-1} \left( \theta - df \right), \qquad \N V = 0.
\end{align*}
The interplay of the two soliton equations and this canonical symmetry has been considered from a physics point of view \cite{papadopoulos2024scale,papadopoulos2025scale}, and observed in recent mathematical works \cite{apostolov2024rigidity,JFS,ivanov2024riemannian,ivanov2023riemannianG2,ivanov2023riemannianspin7}. The main purpose of this paper is to give a general investigation of the structure of string backgrounds by deriving rigidity results in the case $V \equiv 0$, and a dimension reduction principle in the case $V$ is nonvanishing.

We begin with a general discussion of symmetry reduction of generalized Ricci solitons admitting a nontrivial Bismut parallel vector field $V$, without reference to any further special geometric structure.  This dimension reduction principle is implicit in the classic work of Strominger \cite{StromingerSST} although we provide a concrete elementary derivation within for convenience.  Such a $V$ is in particular a Killing field, preserves $H$, and has constant norm, which we will assume without loss of generality is equal to $1$.  We thus obtain a principal connection $\mu = V^{\flat}$ with curvature $F_{\mu} = d \mu$.  As we show in \S \ref{s:DR}, this automatically puts $H$ into the `string ansatz'
\begin{align*}
    H = CS(\mu) + \hat{H},
\end{align*}
where $CS(\mu) = \mu \wedge F_{\mu}$ is the Chern-Simons three form associated to $\mu$, and $\hat{H}$ is some basic form relative to $\mu$.  Moreover, there is a canonically associated horizontal metric $\hat{g}$ and hence a horizontal Bismut connection associated to $(\hat{g}, \hat{H})$.  The three-form $\hat{H}$ will no longer be closed, by rather satisfy the \emph{anomaly cancellation equation/Bianchi identity}
\begin{align*}
    d \hat{H} + F_{\mu} \wedge F_{\mu} = 0.
\end{align*}
If we further impose that $(g, H)$ is a generalized Ricci soliton relative to the vector field $X$, where $V = X - \N f$, then the horizontal data $(\hat{g}, \hat{H})$ satisfies the \emph{gradient string generalized Ricci soliton equations}, i.e.
\begin{align*}
    \Rc^{\hat{\N}} + F_{\mu}^2 + \hat{\N} df =&\ 0,\\
   d^{\star}_{\hat{g}} F - \IP{F, H}_{\hat{g}} + i_{\N f} F =&\ 0.
\end{align*}
Finally, in this setup, there is an elementary but crucial point that for any invariant form preserved by $\N$, the components of its canonical decomposition into basic forms are preserved by $\hat{\N}$.  In particular, this shows that for the various relevant settings $(SU(3), G_2, Spin(7))$, the transverse geometric structures automatically admit a connection with skew-symmetric torsion preserving the structure, a nontrivial condition implying subtle restrictions on the torsion.

The purpose of this paper is to elucidate the fundamental features of the dimension reduction described above, unify and simplify the proofs of several recent results regarding dimension reduction of string backgrounds (cf. \cite{apostolov2024rigidity,fino2025some}), and show that this fundamental pattern naturally extends to new geometric settings such as almost Hermitian, almost contact, and $\Spin(7)$ geometry.  We summarize the main results informally here:
\begin{enumerate}
    \item Establish the canonical symmetry in the almost Hermitian setting (Theorem \ref{t:ahstructure}) and show the transverse almost Hermitian structure is a conformally balanced string generalized Ricci soliton (Theorem \ref{t:ahreduction}).
    \item Establish the canonical symmetry in the almost contact setting (Theorem \ref{t:reducedholonomyAC}), and show the transverse almost contact  structure is a conformally balanced string generalized Ricci soliton (Theorem \ref{t:acreduction}).
    \item Show that when the canonical symmetry for $SU(3)$ structures (cf. \cite{ivanov2024riemannian}) is nontrivial, the complex structure is integrable, hence the structure is Bismut-Hermitian-Einstein (Theorem \ref{t:SU(3)rigidity}) and thus described by a transverse K\"ahler structure (cf. \cite{apostolov2024rigidity}).
    \item Show that when the canonical symmetry for $G_2$ structures (cf. \cite{ivanov2023riemannianG2}) is nontrivial, the transverse $SU(3)$ structure is a conformally balanced string soliton, compute its torsion, and show it is furthermore conformally half-flat (Theorem \ref{t:G2reduction}).  When the symmetry is vanishing we generalize a rigidity observation of \cite{ivanov2023riemannianG2} (Proposition \ref{p:G2rigidity}), and we give a characterization of when the symmetry defines a Riemannian splitting (Corollary \ref{c:G2splitting}).
    \item Show that when the canonical symmetry for $\Spin(7)$ structures (cf. \cite{ivanov2023riemannianspin7}) is nontrivial, the transverse $G_2$ structure is an integrable, constant type, conformally balanced string soliton (Theorem \ref{t:Spin(7)reduction}).  Give a characterization of when the symmetry defines a Riemannian splitting (Corollary \ref{c:spin(7)splitting}).  Interpret the reduced structure as a $G_2$-instanton on an extended tangent bundle (Remark \ref{r:spin7asG2instanton}).
\end{enumerate}

We illustrate the dimension reduction phenomenon with some explicit examples.  Up to this point, the only \emph{compact} examples of Bismut connections with reduced holonomy arise from flat connections, i.e. have trivial holonomy.  However, the dimension reduction yields nontrivial examples of string solitons on homogeneous spaces.  In Example \ref{ex:nonintG2} we exhibit a $G_2$-structure on $SU(2) \times SU(2) \times U(1)$ which yields a reduced $SU(3)$ structure on $SU(2) \times SU(2)$ which is a co-coupled, half flat strong torsion $SU(3)$ structure with reduced holonomy.  Furthermore in Example \ref{ex:nonintG2nonclosedLee} we define a $G_2$ structure again on $SU(2) \times SU(2) \times U(1)$ which yields a co-coupled half-flat SU(3) string soliton on $\mathbb {CP}^1 \times S^3 \times S^1$.  By reducing a Spin(7) structure on $U(1) \times SU(2) \times SU(2) \times U(1)$, we exhibit a $G_2$ string soliton on $\mathbb{CP}^1 \times S^3 \times T^2$ in Example \ref{ex:nonintSpin7OneA}.  Furthermore, in Example \ref{ex:nonintSpin7Two} we exhibit a strong torsion Spin(7) structure on $SU(3)$ whose canonical symmetry generates an $\mathbb R$-action and defines a $G_2$ string soliton on the transverse distribution.

\begin{rmk} We take pains to adopt a consistent naming scheme for the various geometric settings below (almost contact, almost Hermitian, $SU(3)$, $G_2$, $\Spin(7)$).  We add the adjective `skew-torsion' if there is a structure-preserving connection with skew-symmetric torsion.  We will add the adjective `strong torsion' if this torsion three-form is in addition closed.  This results in mostly small modifications of existing terminologies, a notable exception being that the Gray-Hervella class G1 \cite{gray1980sixteen} is here a `skew-torsion almost Hermitian structure.'
\end{rmk}

\section{Dimension reduction, Bismut connections, and generalized Ricci solitons} \label{s:DR}

In this section we give a concrete derivation of fundamental key structural results for generalized Ricci solitons admitting Bismut-parallel vector fields.  Though in a sense these results are not new, we give proofs for concreteness, and refer to \cite{garcia2024pluriclosed} for further background.  A key feature is the natural appearance of a Chern-Simons three form, which has been observed previously in physics literature (cf. \cite{gran2006spinorial,gran2019classification}).

\begin{rmk} In what follows below we will be investigating various geometric structures on a manifold $M$ equipped with a Killing vector field $V$. Unless otherwise specified we do not assume any structure of the quotient space $M / \IP{V}$.  Hence many objects are implicitly defined below only on the transverse distribution $V^{\perp}$, and we will overload various definitions without imposing the further adjective `transverse.'
\end{rmk}

\begin{lemma} \label{l:reducedBC} Suppose $(M^n, g, H)$ is a manifold with closed three form endowed a vector field $V$ such that $\N V = 0$, $\brs{V} = 1$, where $\N = D + \tfrac{1}{2} H$ is the Bismut connection associated to $(g, H)$.  Then, defining $\mu = V^{\flat}$, one has
\begin{enumerate}
    \item $L_V g = L_V H = 0$.
    \item $g = \mu \otimes \mu + \hat{g}$, where $\hat{g}$ is a Riemannian metric on $V^{\perp}$.
    \item $H = CS(\mu) + \hat{H}$, where $\hat{H}$ is basic.
    \item The connection $\hat{\N}$ induced by $\N$ on the horizontal distribution is identified with the Bismut connection associated to $(\hat{g}, \hat{H})$.
    \item Given an $V$-invariant, $\N$-parallel differential form $\varphi$, there exist basic forms $\ga, \gb$ such that $\varphi = \mu \wedge  \ga + \gb$, and $\hat{\N} \ga = 0, \hat{\N} \gb = 0$.
\end{enumerate}
\begin{proof} Note that the condition $\N V = 0$ is equivalent to
\begin{align*}
    L_V g = 0, \qquad d \mu = i_{V} H.
\end{align*}
Taking the exterior derivative of the second equation and using the Cartan formula and $d H = 0$ gives $L_V H = 0$, yielding item (1).

As $V$ is a nonzero Killing field it locally defines the structure of an $\mathbb R$-bundle, with respect to which $g$ admits the usual description $g = a \mu \otimes \mu + \hat{g}$ for a smooth function $a$.  However as $\brs{V} = 1$ it follows that $a = 1$, giving item (2).

To establish item (3), first observe using the invariance of $H$ that it admits a unique expression $H = \mu \wedge  \hat{\ga} +  \hat{H}$ for basic forms $ \hat{\ga},  \hat{H}$.  However, the condition $d \mu = i_V H$ precisely says that $\hat{\ga} = d \mu$, yielding item (3).

To establish item (4), we recall that the connection induced on the horizontal distribution can be expressed
\begin{align*}
    \hat{\N}_X Y = \pi_* \N_{\hat{X}} \hat{Y},
\end{align*}
where $X, Y$ are horizontal vectors and $\hat{X}, \hat{Y}$ denotes the horizontal lift.  It follows easily from item (2) that this connection preserves $\hat{g}$, and from item (3) that the torsion is $\hat{H}$, thus by uniqueness the claim follows.

To establish item (5), we first observe we can define $\ga$ and $\gb$ uniquely via $ \ga = i_V \varphi$, $\gb = \varphi - \mu \wedge  \ga$.  To show the parallelism we establish the case when $\varphi$ is a two-form, with the general case following analogously.  Using that $V$ is parallel we have
\begin{align*}
    (\hat{\N}_X \ga)(Y) =&\ X \ga(Y) - \ga(\hat{\N}_X Y)\\
    =&\ \hat{X}  \ga(\hat{Y}) -  \ga \left( \N_{\hat{X}} {\hat{Y}} \right)\\
    =&\ \hat{X} \varphi(V, \hat{Y}) - \varphi \left( V, \N_{\hat{X}} {\hat{Y}}  \right)\\
    =&\ \hat{X} \varphi(V, \hat{Y}) - \varphi(\N_{\hat{X}} V, \hat{Y}) - \varphi(V, \N_{\hat{X}} \hat{Y})\\
    =&\ (\N_{\hat{X}} \varphi)(V,\hat{X},\hat{Y})\\
    =&\ 0.
\end{align*}
A similar computation shows that $\gb$ is $\hat{\N}$-parallel as well.
\end{proof}
\end{lemma}

\begin{rmk} \label{r:reversibility} The construction of Lemma \ref{l:reducedBC} is reversible in the following sense: fix the data $(\hat{g}, F, \hat{H})$ of a Riemannian metric, closed two-form, and a three form on a manifold such that 
\begin{align*}
    d \hat{H} + F \wedge F = 0.
\end{align*}
Further assuming $[F] \in H^2(\hat{M}, \mathbb Z)$, by Chern-Weil theory we obtain a principal $U(1)$-bundle $\pi: M \to \hat{M}$ together with a principal connection $\mu$ such that $d \mu = F$.  Then, defining $(g, H)$ on $M$ as in items (2) and (3) of Lemma \ref{l:reducedBC}, it follows that $d H = 0$, and moreover the canonical vector field $V$ associated to the action is $\N$-parallel on $M$.  Furthermore given $\ga$ a $\hat{\N}$-parallel form on $\hat{M}$, both $\pi^* \ga$ and $\mu \wedge \pi^* \ga$ are $\N$-parallel.
\end{rmk}

\begin{defn} \label{d:GRS} The data $(g, H, X)$ of a Riemannian metric, closed three-form, and vector field $X$ define a \emph{generalized Ricci soliton} (GRS) if
\begin{align*}
    \Rc^{\N} + \N X^{\flat} = 0.
\end{align*}
We say that the soliton is \emph{gradient} if there exists a function $f$ such that $X = \N f$. 
\end{defn}

\begin{defn} \label{d:stringGRS} The data $(g, H, F,X)$ of a Riemannian metric, three-form, closed two-form and vector field $X$ define a \emph{string generalized Ricci soliton} (string GRS) if
\begin{align*}
    \Rc^{\N} +  F^2 + \N X^{\flat} =&\ 0,\\
   d^{\star}F - \IP{F, H} + i_X F =&\ 0,\\
    d H + F \wedge F =&\ 0.
\end{align*}
We say that the soliton is \emph{gradient} if there exists a function $f$ such that $X = \N f$. 
\end{defn}

\begin{prop} \label{p:reducedGRS} Suppose $(M, g, H, X)$ is a compact GRS.  Then there exists $f$ such that $(g, H, f)$ is a gradient GRS, and the vector field $V = X - \N f$ satisfies $\N V = 0$.  Then either $V = 0$ or without loss of generality we may assume $\brs{V} = 1$.  In this case $(\hat{g}, F_{\mu}, \hat{H},f)$ defines a gradient string GRS on $V^{\perp}$.
\begin{proof} It follows from the gradient formulation of generalized Ricci flow (\cite{OSW}, cf. \cite{GRFbook} Ch. 6) that a compact GRS is automatically a gradient GRS.  The two soliton equations $\Rc^{\N} + \N X^{\flat} = 0 = \Rc^{\N} + \N d f$ immediately imply that $V$ is $\N$-parallel, and in particular has constant norm since $\N$ preserves $g$.  If $V$ is nonvanishing, by scaling we can assume $\brs{V} = 1$, and Lemma \ref{l:reducedBC} applies.  We observe by item (3) that the equation $d H = 0$ is then equivalent to $d \hat{H} + F_{\mu} \wedge F_{\mu} = 0$.  Moreover the setup of \cite{garcia2024pluriclosed} \S 5 applies, and in particular Proposition 5.4 yields the remaining string generalized Ricci soliton equations.
\end{proof}
\end{prop}

We end this section with a further geometric observation, namely that $V$ preserves $f$ and moreover the weighted generalized scalar curvature associated to $X$ is constant.  While we typically scale so that $\brs{V} \equiv 1$ we state the result below in a general scale-invariant fashion.

\begin{prop} Given the setup of Proposition \ref{p:reducedGRS}, one has
\begin{align*}
    V f = 0.
\end{align*}
Furthermore
\begin{align*}
    \gl + \brs{V}^2 =&\ R - \tfrac{1}{12} \brs{H}^2 + 2 \divg X - \brs{X}^2.
\end{align*}
\begin{proof} It follows from (\cite{GRFbook} Proposition 4.33) that in general a GRS satisfies that $R - \tfrac{1}{12} \brs{H}^2 + 2 \gD f - \brs{\N f}^2$ is constant, and this constant must be $\gl$ by integration.  We further observe using Proposition \ref{p:reducedGRS} that
\begin{align*}
    0 = \tr_g L_{X - \N f} g = 2 \divg X - 2 \gD f.
\end{align*}
To show that $V$ preserves $f$ we first note using Lemma \ref{l:reducedBC} that $V$ preserves $g$ and $H$, hence differentiating the constant weighted scalar curvature equation yields
\begin{align*}
    0 = \gD L_V f - \IP{\N L_V f, \N f}.
\end{align*}
Since $M$ is compact, by the maximum principle we conclude $L_V f$ is constant, and this constant can only be zero.  Observe that this fact is equivalent to
\begin{align*}
    0 = V f = (X - \N f) f = \IP{X, \N f} - \brs{\N f}^2.
\end{align*}
Using this it follows that
\begin{align*}
    \brs{V}^2 = \brs{X - \N f}^2 = \brs{X}^2 - 2 \IP{X, \N f} + \brs{\N f}^2 = \brs{X}^2 - \brs{\N f}^2.
\end{align*}
Using the above equations together with the constant scalar curvature equation yields the result.
\end{proof}
\end{prop}

\begin{prop} \label{p:fHrigidity} Given $(M^n, g, H, f)$ a compact GRS, then $f$ is constant if and only if $\brs{H}^2$ is constant.
\begin{proof} It follows from Bianchi identities that GRS automatically satisfy the dilatino equation
\begin{align*}
    \gD f - \brs{\N f}^2 + \tfrac{1}{6} \brs{H}^2 \equiv \gl.
\end{align*}
Since $M$ is compact the result follows easily from the maximum principle.
\end{proof}
\end{prop}

\section{Strong torsion almost-Hermitian structures} \label{s:AH}
\subsection{Definitions and structure theorems}

\begin{defn} \label{d:almhermstrongtorsion} Fix $(M^{2n}, g, J)$ an almost Hermitian manifold.  We say that it defines a \emph{skew-torsion almost Hermitian structure} if there exists a connection $\N$ with skew-symmetric torsion $H_{\gw}$ satisfying $\N g = 0, \N J = 0$.  From (\cite{friedrich2002parallel} Theorem 10.1) such a connection exists if and only if the Nijenhuis tensor is skew-symmetric, in which case the connection is unique and moreover \footnote{Below is the first instance of a general notation where we will decorate $H$ with some piece of the structure to which it is associated, which will be clarifying in the dimension reduction results to come.  While in this setting $H_{\gw}$ of course also depends on $J$, we do not include it to avoid clutter.}
\begin{align} \label{eq:AHHform}
    H_{\gw} = d^c {\gw} + N.
\end{align}
We say that it is a \emph{strong torsion almost Hermitian structure} if $d H_{\gw} = 0$.  We say that the structure has \emph{reduced holonomy} if the holonomy of $\N$ lies in $SU(n)$.
\end{defn}

\begin{rmk} The class of almost Hermitian skew-torsion structures is the same as the Grey-Hervella class $G_1$ \cite{gray1980sixteen}.  We adopt the terminology above to emphasize the analogies with other geometric settings to come.
\end{rmk}

\begin{defn} \label{d:ahcurvatures}
Given $(M^{2n}, g, J)$ a skew-torsion almost Hermitian structure, define the \emph{Lee form} by
\begin{align*}
    \theta_{\gw} = - \tfrac{1}{2} \sum_{i=1}^{2n} H_{\gw}(J \cdot, e_i, J e_i).
\end{align*}
Furthermore we define the \emph{Bismut Ricci form}
\begin{align*}
    \rho(X,Y) = \tfrac{1}{2} \sum_{i=1}^{2n} R^{\N}(X,Y,e_i, J e_i).
\end{align*}
\end{defn}

\begin{rmk}
    The structure has reduced holonomy if and only if $\rho \equiv 0$.  In case $J$ is integrable this is the notion of a Bismut-Hermitian Einstein metric studied in (cf. \cite{apostolov2024rigidity,barbaro2023bismut,barbaro2024pluriclosed,JFS,ye2025bismut}).  Aspects of the dimension reduction discussed below were shown in \cite{apostolov2024rigidity}.
\end{rmk}

\begin{thm} \label{t:ahstructure} (cf. \cite{ivanov2004connections} Theorem 10.5) Let $(M, g, J)$ be a compact strong torsion almost Hermitian manifold with reduced holonomy.  Then
\begin{enumerate}
    \item $\Rc^{\N} +\ \N \theta_{\gw} = 0$,
    \item There exists a unique smooth function $f$ such that $\int_M e^{-f} dV_g = 1$ and $\Rc^{\N} + \N df = 0$,
    \item The vector field $V = \theta_{\gw}^{\sharp} - \N f$ satisfies $\N V = L_V g = 0$, and $\N JV = L_{JV} g = 0$.
\end{enumerate}
\begin{proof}
The proof hinges on a key curvature identity for strong torsion almost Hermitian structures, proved in the Hermitian case in \cite{Ivanovstring} (3.16), and generalized to the almost Hermitian case in (\cite{friedrich2002parallel} Theorem 10.5).  In particular one has
\begin{align*}
    \rho(X, Y) = \left(\Rc^{\N} + \N \theta_{\gw} \right) (X, JY).
\end{align*}
Thus the vanishing of $\rho$ immediately implies (1).  Given (1), item (2) follows from Proposition \ref{p:reducedGRS} above.  For item (3), first note that Proposition \ref{p:reducedGRS} also implies $L_V g = 0$, and moreover since $J$ is parallel $\N JV = 0$, hence also $L_{JV} g = 0$.
\end{proof}
\end{thm}

\begin{rmk} \label{r:BHErigid} We note that in \cite{JFS} Proposition 2.6 it was shown that in the case $J$ is integrable, we have $V = 0$ if and only if the metric is K\"ahler Calabi-Yau.  Moreover a direct proof was given in \cite{apostolov2024rigidity}, relying on a key torsion identity (cf. \cite{Ivanovstring}).  This identity has an extension to the nonintegrable case (\cite{ivanov2024riemannian} 2.21), although in this setting a term involving the Nijenhuis tensor nullifies this argument.  Indeed the rigidity is false: in Example \ref{ex:nonintsu(3)} we observe a strong torsion $SU(3)$ structure with nonintegrable $J$, and $V = 0$.
\end{rmk}

\subsection{Dimension Reduction}

\begin{thm} \label{t:ahreduction} Let $(M^{2n}, g, J)$ be a compact strong torsion almost Hermitian structure with reduced holonomy.  Assume $V$ as in Theorem \ref{t:ahstructure} is nonzero and further assume that $L_V J = L_{JV} J = 0$.  Let $(\hat{g}, \hat{J})$ denote the canonically induced almost Hermitian structure on $\{V, JV\}^{\perp}$.  Then
\begin{enumerate}
    \item $(\hat{g}, \hat{J})$ is a skew-torsion almost Hermitian structure,
    \item $H_{\gw} = CS(\mu) + CS(J \mu) + H_{\hat{\gw}}$, where $H_{\hat{\gw}} = d^c \hat{\gw} + N_{\hat{\gw}}$,
    \item $(\hat{g}, F, H_{\hat{\gw}},f)$ is a string generalized Ricci soliton, where $F = F_{\mu} \otimes V + F_{J\mu} \otimes JV$,
    \item $F_{\mu} = d \theta_{\gw} \in \Lambda^{1,1}_0,\ F_{J \mu} \in \Lambda^{1,1},\ \tr_{\hat{\gw}} F_{J \mu} = -2$,
    \item $\theta_{\hat{\gw}} = df$.
\end{enumerate}
\begin{proof} We first observe that the assumption $L_V J = L_{JV} J = 0$ implies furthermore that $[V, JV] = 0$, hence we have all structure preserved by the rank $2$ abelian distribution $\{V, JV \}$.
It follows that we obtain a canonically induced transverse almost Hermitian structure $(\hat{g}, \hat{J})$ via
\begin{gather} \label{f:inducedAH}
\begin{split}
    \hat{g} =&\ g - \mu \otimes \mu - J \mu \otimes J \mu,\\
    \hat{J} =&\ J - JV \otimes \mu - V \otimes \mu J,\\
    \hat{\gw} =&\ \gw - \mu \wedge J \mu.
\end{split}
\end{gather}
It follows by an obvious extension of Lemma \ref{l:reducedBC}(5) that $\hat{\N} \hat{\gw} = 0$, hence $\hat{\N} \hat{J} = 0$, hence item (1) follows.  Furthermore item (2) follows from Lemma \ref{l:reducedBC}(3) together with the uniqueness of the skew-torsion connection.  Item (3) follows from Proposition \ref{p:reducedGRS}.  For item (4) we first note that the fact that $V$ and $JV$ preserve $J$ implies that $F_{\mu}$ and $F_{J \mu}$ are type $(1,1)$.  The claims on their traces as well as item (5) follows from the proof of \cite{apostolov2024rigidity} Proposition 2.10, as it does not rely on integrability of $J$.
\end{proof}
\end{thm}

\section{Strong torsion HKT structures} \label{s:HKT}

In this section we will address the case of strong torsion hyperHermitian structures.  Other than the derivation of the string soliton equation for the transverse geometry in Theorem \ref{t:hktreduction}, the results are mild extensions of the recent works \cite{apostolov2024rigidity, brienza2025structure}, and very similar in spirit to the previous section, so we will be brief.

\subsection{Definitions and structure theorems}

\begin{defn} \label{d:HKTstrongtorsion} A \emph{hypercomplex manifold} is a quadruple $(M^{4n}, I, J, K)$ a smooth manifold with integrable complex structures satisfying the quaternion relations.  A metric $g$ on $M$ is \emph{hyperHermitian} if it is compatible with $I, J$, and $K$.  We say that it defines a \emph{skew-torsion hyperHermitian structure} if there exists a connection $\N$ with skew-symmetric torsion $H_g$ satisfying $\N g = 0, \N I = 0, \N J = 0, \N K = 0$, in which case the connection is unique and agrees with the unique Bismut connections associated to the pairs $(g, I), (g, J)$, and $(g, K)$.  In other words,
\begin{align*}
    H_g = d^c_I \gw_I = d^c_J \gw_J = d^c_K \gw_K.
\end{align*}
Furthermore, it follows from \cite{fusi2024special} that the Lee forms of the Hermitian structures agree, i.e. there exists $\theta_g$ such that
\begin{align*}
    \theta_g = \theta_{\gw_I} = \theta_{\gw_J} = \theta_{\gw_K}.
\end{align*}
Finally, we say that the structure is a \emph{strong torsion hyperHermitian structure} if $d H_g = 0$.  It follows from \cite{Ivanovstring} that for a skew-torsion hyperHermitian structure, each underlying Hermitian structure has vanishing Bismut Ricci curvature, and moreover the holonomy of $\N$ lies in $Sp(n)$.
\end{defn}

\begin{thm} \label{t:hktstructure} (cf. \cite{JFS} Proposition 2.6 \cite{apostolov2024rigidity} Proposition 2.6, \cite{brienza2025structure} Proposition 2.6) Let $(M, g, J)$ be a compact strong torsion hyperHermitian structure.  Then
\begin{enumerate}
    \item $\Rc^{\N} +\ \N \theta_{g} = 0$,
    \item There exists a unique smooth function $f$ such that $\int_M e^{-f} dV_g = 1$ and $\Rc^{\N} + \N df = 0$,
    \item The vector field $V = \theta_{g}^{\sharp} - \N f$ satisfies $\N V = \N IV = \N JV = \N KV = 0$, and $L_V I = L_V J = L_V K = 0$.
\end{enumerate}
\begin{proof} Item (1) follows precisely as in item (1) of Theorem \ref{t:ahstructure} (it is even a special case of that claim).  Given (1), item (2) follows from Proposition \ref{p:reducedGRS} above.  It is clear that $V$ is parallel, hence so are $IV, JV$, and $KV$ since the connection preserves the hypercomplex structure.  Furthermore, it follows from the results cited in \cite{JFS, apostolov2024rigidity} that for the BHE structure $(g, I)$ the vector field $V_I = \theta_I^{\sharp} - \N f$ satisfies $L_{V_I} I = 0$, where the function $f$ is the unique minimizer of the $\FF$-functional as explained in Proposition \ref{p:reducedGRS}.  Since $f$ as constructed is independent of the complex structure and the Lee forms all agree, it follows that $V_I = V_J = V_K = V$, and so $L_V I = L_V J = L_V K = 0$ as claimed.
\end{proof}
\end{thm}

\subsection{Dimension Reduction}

\begin{thm} \label{t:hktreduction} Let $(M, g, J)$ be a compact strong torsion hyperHermitian structure.  Assume $V$ as in Theorem \ref{t:ahstructure} is nonzero, and that $\{V, IV, JV, KV \}$ span an integrable distribution.  Let $\hat{g}$ denote the canonically induced metric on $\{V, IV JV, KV\}^{\perp}$.  Then,
\begin{enumerate}
    \item $\{V, IV, JV, KV\}$ generate a Lie algebra isomorphic to either $\mathfrak t^4$ or $\mathfrak u(2)$.
    \item Setting $\bar{\mu} = \mu \otimes V + I \mu \otimes IV + J \mu \otimes JV + K \mu \otimes KV$, we have $H_{g} = CS(\bar{\mu}) + H_{\hat{g}}$.
    \item $(\hat{g}, F_{\bar{\mu}}, H_{\hat{g}},f)$ is a string generalized Ricci soliton.
\end{enumerate}
\begin{proof} The proof of (\cite{brienza2025structure} Theorem 5.3) shows both items (1) and (2) in every dimension using the properties of $V$, though ultimately the proof is concerned with dimension $8$ (see Remark \ref{r:HKTremark} below).  For item (3) we observe that the proof of Proposition \ref{p:reducedGRS} easily extends to the case of a nonabelian Lie algebra as the computations of \cite{garcia2024pluriclosed} \S 5 are carried out in this more general case.
\end{proof}
\end{thm}

\begin{rmk} \label{r:HKTremark} In \cite{brienza2025structure} Theorem 5.3 it is shown that in dimension $8$ the relevant Lie algebra is always $\mathfrak u(2)$, and it is conjectured that this is true in all dimensions.  Given this, it follows from \cite{pedersen1998hypercomplex} Theorem 2.1 that the horizontal distribution admits a canonical quaternionic structure.  In this case it follows from the discussion here that it is moreover a skew-torsion quaternionic structure (cf. \cite{ivanov2002geometry}), in the sense that it admits a compatible connection with skew-symmetric torsion.  Moreover by Theorem \ref{t:hktreduction}(2) this torsion satisfies an anomaly cancellation equation.
\end{rmk}

\section{Strong torsion almost contact structures} \label{s:AC}
\subsection{Definitions and structure theorems}

We first recall the fundamental setup of almost contact metric manifolds (cf. \cite{blaircontact}).

\begin{defn} \label{d:almostcontactmetric} We say that $(M, g, \xi, \varphi)$ is an \emph{almost contact metric manifold} if $g$ is a Riemannian metric, $\xi$ is a vector field and $\varphi$ is an endomorphism such that, setting $\eta = \xi^{\flat}$,
\begin{align*}
    \brs{\xi} = 1, \quad \varphi^2 = - \Id +\ \eta \otimes \xi, \quad g(\varphi(X), \varphi(Y)) = g(X,Y) - \eta(X) \eta(Y), \quad \varphi(\xi) = 0.
\end{align*}
Such structures induce an almost Hermitian structure on $\xi^{\perp}$, with fundamental form $\gw(X,Y) = g(X, \varphi Y)$.  Furthermore, the endomorphism $\varphi$ has an associated Nijenhuis tensor (n.b. this is a modification of the general notion of the Nijenhuis tensor of an endomorphism)
\begin{align*}
    N(X, Y) = [\varphi X,\varphi Y] + \varphi^2[X,Y] - \varphi[\varphi X, Y] - \varphi[X,\varphi Y] + d \eta(X,Y) \xi.
\end{align*}
\end{defn}

\begin{defn} \label{d:almostcontacttorsion} We say that an almost contact metric manifold $(M^{2n+1}, g, \xi, \varphi)$ defines a \emph{skew-torsion almost contact structure} if there exists a connection $\N$ with skew-symmetric torsion $H_{\xi}$ satisfying $\N g = 0, \N \xi = 0, \N \varphi = 0$.  It follows from (\cite{friedrich2002parallel} Theorem 8.2) that this connection exists if and only if the Nijenhuis tensor is skew-symmetric and $\xi$ is a Killing field, in which case it is unique and one has
\begin{align*}
    H_{\xi} = \eta \wedge d \eta + d^{\varphi} \gw + N^{\perp},
\end{align*}
where $N$ is uniquely decomposed as
\begin{align*}
    N = \eta \wedge i_{\xi} N + N^{\perp},
\end{align*}
and
\begin{align*}
    d^{\varphi} \gw(X,Y,Z) = - d \gw(\varphi X, \varphi Y, \varphi Z).
\end{align*}
We say that structure is a \emph{strong torsion almost contact structure} if $d H_{\xi} = 0$.  Finally, we say that a strong almost contact torsion structure has \emph{reduced holonomy} if the holonomy of $\N$ lies in $SU(n)$.
\end{defn}

\begin{defn} \label{d:accurvatures} Given $(M^{2n+1}, g, \xi, \varphi)$ a skew-torsion almost contact torsion structure, define the \emph{Lee form} by
\begin{align*}
    \theta_{\xi} = - \tfrac{1}{2} \sum_{i=1}^{2n+1} H_{\xi}(\varphi \cdot, e_i, \varphi e_i).
\end{align*}
Furthermore we define the \emph{Bismut Ricci form}
\begin{align*}
    \rho(X,Y) = \tfrac{1}{2} \sum_{i=1}^{2n+1} R^{\N}(X,Y,e_i, \varphi e_i).
\end{align*}
\end{defn}

\begin{rmk} As explained in \cite{friedrich2002parallel} \S 9, the structure of an almost contact torsion structure on $M^{2n+1}$ automatically implies that the holonomy lies in $U(n)$, due to the parallelism of $\xi$ and $\varphi$.  A further reduction of the holonomy to $SU(n)$ is equivalent to the vanishing of the associated Bismut Ricci form.
\end{rmk}

\begin{thm} \label{t:reducedholonomyAC} Let $(M^{2n+1}, g, \xi, \varphi)$ be a strong torsion almost contact structure with reduced holonomy.  Then
\begin{enumerate}
    \item $\Rc^{\N} + \N \theta_{\xi} = 0$,
    \item There exists a unique smooth function $f$ such that $\int_M e^{-f} dV_g = 1$ and $\Rc^{\N} + \N df = 0$,
    \item The vector field $V = \theta_{\xi}^{\sharp} - \N f$ satisfies $\N V = L_V g = 0, \N \varphi V = L_{\varphi V} g = 0$.
\end{enumerate}
\begin{proof} First recall a curvature identity for almost contact metric structures which bears a subtle relationship to (\cite{friedrich2002parallel} Theorem 10.5).  Let
\begin{align*}
    \ga_{\xi} := - \tfrac{1}{2} \sum_{i=1}^{2n+1} H(\cdot, e_i, \varphi(e_i)).
\end{align*}
This tensor is closely related to $\varphi \theta_{\xi}$, although care must be taken with the direction $\xi$.  It follows from (\cite{friedrich2002parallel} Proposition 9.1) that
\begin{align*}
    \rho(X, Y) = \Rc^{\N} (X, \varphi Y) - (\N_X \ga)(Y).
\end{align*}
Using this we observe that for $Z = \varphi Y$, $g(Y, \xi) = 0$,
\begin{align*}
    \left(\Rc^{\N} + \N \theta_{\xi} \right) (X, Z) =&\ \Rc^{\N} (X, \varphi Y) + \N_X \theta_{\xi}(\varphi Y)\\
    =&\ \Rc^{\N}(X, \varphi Y) - \N_X \ga(Y) \\
    =&\ \rho(X, Y) = 0.
\end{align*}
Furthermore, it follows trivially from the definition of $\theta_{\xi}$ that $\theta_{\xi}(\xi) = 0$.  Since $\N \xi = 0$ it follows that $\N_X \theta_{\xi}(\xi) = 0$ for any $X$.  Also since $\xi$ is $\N$-parallel, it follows that $\Rc^{\N}(X, \xi) = 0$, hence $(\Rc^{\N} + \N \theta_{\xi})(X, \xi)= 0$, finishing the proof of item (1).  Given item (1), item (2) follows from Proposition \ref{p:reducedGRS}, and item (3) follows easily.
\end{proof}
\end{thm}

\subsection{Dimension Reduction}

\begin{thm} \label{t:acreduction} Let $(M , g, \varphi)$ be a compact strong torsion almost contact structure with reduced holonomy.  Assume $V$ as in Theorem \ref{t:reducedholonomyAC} is nonzero.  Then $\IP{V, \xi} = 0$, hence $\{V, \varphi V\}$ span a rank $2$ distribution of $\N$-parallel vector fields.  Let $(\hat{g}, \hat{\varphi},\xi)$ denote the canonically induced almost contact structure on $\{V, \varphi V\}^{\perp}$.  Then
\begin{enumerate}
    \item $(\hat{g}, \hat{J})$ is a skew-torsion almost contact structure,
    \item $H_{\xi} = CS(\mu) + CS(\varphi \mu) + H_{\hat{\xi}}$, where $H_{\hat{\xi}} = d^{\hat{\varphi}} \hat{\gw} + \hat{N}_{\hat{\varphi}}^{\perp}$,
    \item $(\hat{g}, F, H_{\hat{\xi}},f)$ is a string generalized Ricci soliton, where $F = F_{\mu} \otimes V + F_{\varphi\mu} \otimes \varphi V$,
    \item $\tr_{\hat{\gw}} F_{\mu} = 0, \tr_{\hat{\gw}} F_{\varphi \mu} = -2$,
    \item $\theta_{\hat{\gw}} = df$.
\end{enumerate}
\begin{proof} 
We first show $\IP{V, \xi} = 0$.  Since $\varphi \xi = 0$ it follows from the definition of $\theta_{\xi}$ that $\IP{\theta_{\xi}^{\sharp}, \xi} = 0$.  Since $V$ and $\xi$ are both $\N$-parallel, $\IP{V, \xi}$ is constant.  Since $M$ is compact, there exists $p \in M$ a critical point for $f$.  At such a point $\IP{V, \xi} = \IP{\theta_{\xi}^{\sharp}, \xi} = 0$, hence it vanishes everywhere as claimed.  Given this, we obtained the induced almost contact structure $(\hat{g}, \hat{\varphi}, \xi)$ by combining the induced almost Hermitian structure on $\{V, JV, \xi\}^{\perp}$ defined as in (\ref{f:inducedAH}) and enforcing $\brs{\xi} \equiv 1$.

It follows by an obvious extension of Lemma \ref{l:reducedBC}(5) that the induced connection $\hat{\N}$ satisfies $\hat{\N} \hat{\varphi} = 0$, hence item (1) follows.  Furthermore by uniqueness of the canonical connection and Lemma \ref{l:reducedBC}(3), we obtain item (2).  Item (3) follows from Proposition \ref{p:reducedGRS}.  Items (4) and (5) follow by tracing the formula for $H_{\xi}$ and decomposing into horizontal and vertical parts as in \cite{apostolov2024rigidity} Proposition 2.10.
\end{proof}
\end{thm}

\section{Strong torsion \texorpdfstring{$SU(3)$}{SU(3)} structures} \label{s:SU(3)}

In this section we begin by recalling fundamental definitions regarding $SU(3)$ structures, their torsion, the basic structure theory of those admitting a connection with skew-symmetric torsion, and the result of Ivanov-Stanchev \cite{ivanov2024riemannian} showing that those with strong torsion are automatically generalized Ricci solitons and admit a canonical symmetry (Theorem \ref{t:SU(3)structure}).  Next in Theorem \ref{t:SU(3)rigidity} we show that when the canonical symmetry is present, the complex structure is automatically integrable, hence the structure is Bismut Hermitian-Einstein, and described in detail by the structure theory of \cite{apostolov2024rigidity}.

\subsection{Definitions and structure theorems}

\begin{defn} \label{su(3)-Structure}
We say that $(M^6, \omega, \Omega^{+})$ is an \emph{$SU(3)$ structure} if $\omega\in\Lambda^{2}$ and $\Omega^{+}\in \Lambda^{3}$ may pointwise be identified with the following forms $\omega_{0}$ and $\Omega_{0}^{+}$ on $\mathbb{R}^{6}$:
\begin{align}\label{f:SU3Pointwise}
\omega_{0} &= e^{12} + e^{34} + e^{56}, \quad 
\Omega_{0}^{+} = e^{135} - e^{146} - e^{236} - e^{245}.
\end{align}
Such forms automatically satisfy
\begin{gather*}
\begin{split}
\omega\wedge\Omega^{+} &= 0, \qquad 
\omega\wedge\omega\wedge\omega = \tfrac{3}{2}\Omega^{+}\wedge\Omega^{-},
\end{split}
\end{gather*}
where the tensor $\Omega^-$ is obtained as follows.  First we define an associated Riemannian metric $g_{\gw}$ via
\begin{align*}
    g_{\gw}(X,Y) \tfrac{\gw^3}{6} = - \tfrac{1}{2} (i_X \gw) \wedge (i_Y \Omega_+) \wedge \Omega_+,
\end{align*}
and then $\Omega^{-} := \star_{g_{\gw}} \Omega^+$.
Moreover, it follows that $(g_{\gw}, \gw)$ are the metric and fundamental form of an almost Hermitian structure, with almost complex structure denoted $J_{\gw}$, and then $\Omega^{\pm}$ are type $(3,0) + (0,3)$.\end{defn} 

\begin{rmk} The forms $(\omega, \Omega^{+})$ underlying an $SU(3)$ structure are stable in the sense of Hitchin, namely, they pointwise lie in open orbits of $GL(V))$ on $\Lambda^{2}(V)$ and $\Lambda^{3}(V)$, respectively. Such globally-defined stable forms in general are associated with $G$-structures \cite{hitchin2001stable}.
\end{rmk}

The space of one-forms is an irreducible representation of $SU(3)$, while the space of two and three-forms decompose into irreducible representations as follows:
\begin{align*}
\Lambda^{2} &= \Lambda_{1}^2 \oplus \Lambda_{6}^2 \oplus \Lambda_{8}^2, \qquad 
\Lambda^{3} = \Lambda_{1\oplus 1}^{3} \oplus \Lambda_{6}^{3} \oplus \Lambda_{12}^{3}.
\end{align*}
These representations admit explicit descriptions:
\begin{gather} \label{f:SU(3)reps}
\begin{split}
\Lambda_{1}^{2} &= \{ a \omega\ |\ a\in C^{\infty}(M)\}\\
\Lambda_{6}^{2} &= \{\star(\alpha\wedge\Omega^{+})\ |\ \alpha \in \Lambda^{1}\}\\
\Lambda_{8}^{2} &= \{\alpha\in\Lambda^{2}\ |\ \alpha\wedge\omega\wedge\omega=0,\ \alpha\wedge\Omega^{+}=0 \}\\
\Lambda_{1\oplus1}^{3} &= \{a_+ \Omega^{+} + a_- \Omega^{-}\ |\ a_{\pm} \in C^{\infty}(M)\}\\
\Lambda_{6}^{3} &= \{\alpha\wedge\omega\ |\ \alpha\in\Lambda^{1}\}\\
\Lambda_{12}^{3} &= \{\alpha\in\Lambda^{3}\ |\ \alpha\wedge\omega=0, \alpha\wedge\Omega^{\pm} = 0\}.
\end{split}
\end{gather}

\begin{defn} Given an $SU(3)$ structure $(\gw, \Omega^+)$ on a 6-manifold, following \cite{chiossisalamon2002} we define torsion forms $\gs_i, \nu_i, \pi_i$ via
\begin{gather*}
\begin{split}
d\omega &= -\tfrac{3}{2}\sigma_{0}\Omega^{+} + \tfrac{3}{2}\pi_{0} \Omega^{-} + \nu_{1}\wedge\omega + \nu_{3},\\
d\Omega^{+} &= \pi_{0}\omega^2 + \pi_{1}\wedge\Omega^{+} - \pi_{2}\wedge\omega,\\
d\Omega^{-} &= \sigma_{0}\omega^2 + (J\pi_{1})\wedge\Omega^{+} - \sigma_{2}\wedge\omega.
\end{split}
\end{gather*}
Furthermore, the \emph{Lee form} of the $SU(3)$ structure is $\theta_{\gw} := \nu_1$.  An $SU(3)$ structure is \emph{balanced} if $\theta_{\gw} = 0$, \emph{half-flat} if $\pi_{0}=0, \nu_{1}=\pi_{1}=0, \textnormal{and}\hspace{2mm} \pi_{2}=0$, and \emph{co-coupled} if it is half-flat and $\sigma_{2}=0$.
\end{defn}

\begin{defn} \label{d:strongSU(3)T} We say $(M, \gw, \Omega^+)$ is a \emph{skew-torsion $SU(3)$ structure} if
there exists a connection $\N$ with skew-symmetric torsion $H_{\gw}$ satisfying $\N \gw = 0, \N \Omega = 0$.  It follows from (\cite{friedrich2002parallel} Theorem 10.1) that in this case the connection is unique and moreover
    \begin{align} \label{f:SU(3)Bismuttorsion}
    H_{\gw} = d^c \gw + N,
    \end{align}
    where $N$ is the Nijenhuis tensor of the almost complex structure $J_{\gw}$, which in particular is thus skew-symmetric.  We say it is a \emph{strong torsion $SU(3)$ structure} if furthermore $d H_{\gw} = 0$.
\end{defn}

\begin{thm} \label{t:SU(3)structure} (cf. \cite{ivanov2024riemannian} Theorems 6.1, 6.5) Let $(M, \gw, \Omega^+)$ be a compact strong torsion $SU(3)$ structure.  Then
\begin{enumerate}
    \item $\Rc^{\N} +\ \N \theta_{\gw} = 0$,
    \item $d \theta_{\gw} \in \Lambda^{1,1}_0$,
    \item $\N N = 0$,
    \item There exists a unique smooth function $f$ such that $\int_M e^{-f} dV_g = 1$ and $\Rc^{\N} + \N df = 0$,
    \item The vector field $V = \theta_{\gw}^{\sharp} - \N f$ satisfies $\N V = L_V \gw = L_V J = L_V \Omega^{\pm} = 0$, and also $\N JV = L_{JV} \gw = L_{JV} J = L_{JV} \Omega^{\pm} = 0$.
\end{enumerate}
\end{thm}

\begin{rmk} 
Items (1),(2) and (3) are (\cite{ivanov2024riemannian} Theorem 6.1).  Items (4) and (5) are contained in (\cite{ivanov2024riemannian} Theorem 6.5) which is stated as an equivalence although case (a) of that theorem holds a priori as discussed above.  Note also item (4) is explained in Proposition \ref{p:reducedGRS} above.
\end{rmk}


\begin{thm} \label{t:SU(3)rigidity} Let $(M, \gw, \Omega^+)$ be a compact strong torsion $SU(3)$ structure such that $V$ is nonvanishing.  Then $J_{\gw}$ is integrable, and hence $(\gw, J_{\gw})$ defines a Bismut Hermitian-Einstein structure.
\begin{proof} Using Theorem \ref{t:ahreduction} we obtain a canonical skew-torsion almost Hermitian structure $(\hat{g}, \hat{J})$ on $\{V, JV\}^{\perp}$.  As this distribution is four-dimensional, and we know ${N}_{\hat{\gw}}$ is type $(3,0) + (0,3)$ and skew symmetric, it follows that ${N}_{\hat{\gw}} = 0$.   Now from Theorem \ref{t:ahreduction} item (1) we have
\begin{align*}
    H_{\gw} = d^c \gw + N = CS(\mu) + \hat{H}.
\end{align*}
Moreover, since $F_{\mu}, F_{J\mu} \in \Lambda^{1,1}$ it follows that $CS(\mu) + CS(J \mu) \in \Lambda^{2,1 + 1,2}$.
Hence
\begin{align*}
    \tfrac{1}{4} N = \pi_{3,0 + 0,3} H_{\gw} = \pi_{3,0 + 0,3} \hat{H} = \tfrac{1}{4} \hat{N} = 0.
\end{align*}
Hence $J$ is integrable and thus $(g, H)$ is Bismut-Hermitian-Einstein.
\end{proof}
\end{thm}

A detailed study of BHE threefolds with $V$ nonvanishing was undertaken in \cite{apostolov2024rigidity}, describing the structure in terms of a canonical transverse K\"ahler structure satisfying a scalar PDE.  As noted above in Remark \ref{r:BHErigid}, in the BHE case $V = 0$ if and only if the structure is K\"ahler Calabi-Yau.  We show in the next example that this rigidity fails in the nonintegrable case.

\begin{ex} \label{ex:nonintsu(3)} (cf. \cite{fino2025some} Ex. 5.8) Consider $S^3 \times S^3 \cong SU(2) \times SU(2)$ with the usual left-invariant coframe $\{e^i\}$ satisfying
\begin{align*}
    d e^1 = -2 e^{23},\ d e^2 = -2 e^{31},\ d e^3 = - 2 e^{12},\ d e^4 = - 2 e^{56},\ d e^5 = - 2 e^{64},\ d e^6 = - 2 e^{45}.
\end{align*}
Next define a new coframe
\begin{align*}
    \eta^1 =&\ \tfrac{1}{2} (e^1 - e^4),\ \eta^2 = \tfrac{1}{2} (e^1 + e^4),\ \eta^3 = \tfrac{1}{2} (e^2 - e^5),\\ \eta^4 =&\ \tfrac{1}{2} (e^2 + e^5),\ \eta^5 = \tfrac{1}{2} (e^3 - e^6),\ \eta^6 = \tfrac{1}{2} (e^3 + e^6).
\end{align*}
We then define an $SU(3)$ structure via
\begin{align*}
    \gw =&\ \eta^{12} + \eta^{34} + \eta^{56},\\
    \Omega^+ + \i \Omega^{-} =&\ \left( \eta^1 + \i \eta^2 \right) \wedge (\eta^3 + \i \eta^4) \wedge (\eta^5 + \i \eta^6).
\end{align*}
For this structure the only nonvanishing torsion components are 
\begin{align*}
    \gs_0 = -2,\quad \nu_3 = 3 \eta^{135} + \eta^{146} + \eta^{236} + \eta^{245}.
\end{align*}
Using this computation it follows further that $d H_{\gw} = 0$, hence this defines a strong torsion $SU(3)$ manifold.
In particular the Nijenhuis tensor is nonvanishing, while the Lee form $\theta_{\gw} = \nu_1 = 0$.  Moreover, due to the left-invariance, the soliton potential $f$ is also left-invariant, hence constant.  It follows that $V = 0$.
\end{ex}

\section{Strong torsion \texorpdfstring{$G_2$}{G2} structures} \label{s:G2}

In this section we begin by recalling fundamental definitions regarding $G_2$ structures, their torsion, the basic structure theory of those admitting a connection with skew-symmetric torsion, and the result of Ivanov-Stanchev \cite{ivanov2023riemannianG2} showing that those with strong torsion are automatically generalized Ricci solitons and admit a canonical symmetry (Theorem \ref{t:g2structure}).  We then in Proposition \ref{p:G2rigidity} extend a rigidity observation of \cite{ivanov2023riemannianG2}, showing precisely how torsion-free $G_2$ structures enter the theory.  Next in Theorem \ref{t:G2reduction} we extend (\cite{fino2025some} Theorem 6.1) and show the transverse $SU(3)$ structure is skew-torsion, and a string generalized Ricci soliton.  Moreover we compute its torsion, showing that it is conformally half-flat, and balanced, with conformal factor $e^{-f}$.  We end by extending a rigidity phenomena of \cite{fino2025some}.  For further recent works on strong torsion $G_2$ structures see \cite{cleyton2021metric, fino2023twisted, lotay2024coupled, moroianu2025mathrm}. 

\subsection{Definitions and structure theorems}

\begin{defn}\label{d:G2Structure} We say a three-form $\varphi$ on a 7-manifold $M$ defines a \emph{$G_{2}$ structure} if $\varphi$ may be pointwise identified with the form $\varphi_{0}$ on $\mathbb{R}^{7}$:
\begin{equation}\label{f:G2pointwise}
\varphi_{0} = e^{127} + e^{135} - e^{146} - e^{236} - e^{245} + e^{347} + e^{567}.
\end{equation}
Such a form $\varphi\in\Lambda^{3}(M)$ is called \emph{positive} and is stable in the sense of Hitchin. The three-form $\varphi$ algebraically determines a Riemannian metric via
\begin{align*}
    g_{\varphi}(X,Y)\textnormal{vol} = \tfrac{1}{6} X\lrcorner\varphi\wedge Y\lrcorner\varphi\wedge\varphi.
\end{align*}
\end{defn}

Fixing a $G_2$ structure, we may decompose the spaces of differential forms into irreducible representations of $G_2$ as follows:
\begin{align*}
\Lambda^{1} = \Lambda_{7}^{1}, \quad 
\Lambda^{2} = \Lambda_{7}^{2} \oplus \Lambda_{14}^{2}, \quad 
\Lambda^{3} = \Lambda_{1}^{3} \oplus \Lambda_{7}^{3} \oplus \Lambda_{27}^{3},
\end{align*}
where
\begin{align*}
\Lambda_{7}^{2} &= \{\alpha\in\Lambda^{2}(M) | \alpha\wedge\varphi = 2\star\alpha\} \\
\Lambda_{14}^{2} &= \{\alpha\in\Lambda^{2}(M) | \alpha\wedge\varphi = -\star\alpha\}\\
\Lambda_{1}^{3} &= \{f\varphi | f\in C^{\infty}(M)\}\\
\Lambda_{7}^{3} &= \{X\lrcorner\star_{\varphi}\varphi | X\in\Lambda^{1} \}\\
\Lambda_{27}^{3} &= \{\gamma\in\Lambda^{3}(M) | \gamma\wedge\varphi = 0, \gamma\wedge\star_{\varphi}\varphi  =0\}.
\end{align*}

\begin{defn} Given a $G_2$ structure $\varphi$, following Fernandez and Gray \cite{fernandez1982riemannian}, we define torsion forms $\tau_i$ via the irreducible components of $d \varphi$ and $d (\star_{\varphi} \varphi)$:
\begin{align*}
d\varphi &= \tau_{0} (\star_{\varphi} \varphi) + 3\tau_{1}\wedge\varphi + \star\tau_{3}\\
d (\star_{\varphi} \varphi) &= 4\tau_{1}\wedge (\star_{\varphi} \varphi) + \tau_{2}\wedge\varphi.
\end{align*}
\end{defn}

\begin{defn} \label{d:strongG2T} We say $(M, \varphi)$ is a \emph{skew-torsion $G_2$ structure} if
there exists a connection $\N$ with skew-symmetric torsion $H_{\varphi}$ satisfying $\N g_{\varphi} = 0, \N \varphi = 0$.  It follows from (\cite{friedrich2002parallel} Theorem 4.8) that such a connection exists if and only if $\tau_{2}=0$ and in this case the connection is unique and moreover
    \begin{align} \label{f:G2torsion}
    H_{\varphi} = - \star_{\varphi} d \varphi + \star_{\varphi} (\theta_{\varphi} \wedge \varphi) + \tfrac{1}{6} \IP{d \varphi, \star_{\varphi} \varphi} \varphi,
    \end{align}
    where
    \begin{align*}
    \theta_{\varphi} = 4 \tau_1
    \end{align*}
    is the \emph{Lee form} of $\varphi$ (cf. \cite{bryant1987metrics}).
    We say the structure is a \emph{strong torsion $G_2$ structure} if furthermore $d H_{\varphi} = 0$.
\end{defn}

\begin{thm} \label{t:g2structure} (cf. \cite{ivanov2023riemannianG2} Theorem 7.1, 7.3) Let $(M, \varphi)$ be a compact strong torsion $G_2$-torsion manifold.  Then
\begin{enumerate}
    \item $\Rc^{\N} +\ \N \theta_{\varphi} = 0$,
    \item $\tau_0$ is constant,
    \item $d \theta_{\varphi} \in \Lambda_{14}^2$,
    \item There exists a unique smooth function $f$ such that $\int_M e^{-f} dV_g = 1$ and $\Rc^{\N} + \N df = 0$,
    \item The vector field $V = \theta_{\varphi}^{\sharp} - \N f$ satisfies $\N V = L_V \varphi = 0$,
\end{enumerate}
\end{thm}

\begin{proof} Items (1) and (2) are (\cite{ivanov2023riemannianG2} Theorem 7.1).  Items (4) and (5) are contained in (\cite{ivanov2023riemannianG2} Theorem 7.3) which is stated as an equivalence although case (a) of that theorem holds a priori as discussed above.  Note also item (4) is explained in Proposition \ref{p:reducedGRS} above.  Item (3) is contained in \cite{karigiannis2005deformations}, but we include the elementary derivation here.  First, the condition $d \theta_{\varphi} \in \Lambda^2_{14}$ is equivalent to $d \theta_{\varphi} \wedge (\star_{\varphi} \varphi) = 0$.  The integrability of $\varphi$ implies
\begin{align} \label{f:LeeformG2}
d (\star_{\varphi} \varphi) = \theta_{\varphi} \wedge (\star_{\varphi} \varphi).
\end{align}
It follows that
\begin{align*}
    d \theta_{\varphi} \wedge (\star_{\varphi} \varphi) = d \left( \theta_{\varphi} \wedge \star_{\varphi} \varphi \right) + \theta_{\varphi} \wedge d (\star_{\varphi} \varphi) = d d (\star_{\varphi} \varphi) + \theta_{\varphi} \wedge \theta_{\varphi} \wedge (\star_{\varphi} \varphi) = 0,
\end{align*}
as claimed.
\end{proof}

We note that (\cite{fino2025some} Corollary 3.3) shows that if $\tau_3$ vanishes then the strong torsion $G_2$ structure is torsion-free.  Furthermore, (\cite{fino2025some} Corollary 3.4) shows by a local argument that the vanishing of $\tau_0$ and $\tau_1$ implies the structure is torsion-free.  Also in (\cite{ivanov2023riemannianG2} Theorem 7.3) it was shown that for a compact strong torsion $G_2$ structure if $\tau_0$ and the canonical vector field $V$ both vanish, the structure is torsion free.  We next show a general integral identity underlying that proof that yields an interesting general relationship between $\tau_0$ and the $\FF$-functional.

\begin{prop} \label{p:G2rigidity} (cf. \cite{ivanov2023riemannianG2} Theorem 7.3) Let $(M, \varphi)$ be a compact strong torsion $G_2$ manifold.  If $V = 0$ then
\begin{align*}
    \int_M e^{-f} \brs{H_{\varphi}}^2 dV_{g_{\varphi}} =&\ \tfrac{49}{36}\tau_0^2 \int_M e^{-f} dV_{g_{\varphi}}.
\end{align*}
In particular $\tau_0 = 0$ if and only if $H_{\varphi}$ is zero and $\varphi$ defines a parallel $G_2$-structure.
\begin{proof} Plugging $V = 0$ into (\ref{f:G2torsion}) yields
\begin{align*}
    H_{\varphi} =&\ - \star_{\varphi} d \varphi + \star_{\varphi} (df \wedge \varphi) + \tfrac{7}{6} \tau_0 \varphi = - \star_{\varphi} e^f d (e^{-f} \varphi) + \tfrac{7}{6} \tau_0 \varphi.
\end{align*}
Also using (\ref{f:G2torsion}) and noting that $\IP{\star_{\varphi}(\theta \wedge \varphi), \varphi} = 0$ and $\IP{\star_{\varphi} d \varphi, \varphi} = 7 \tau_0$, it follows that $\IP{H_{\varphi}, \varphi} = \tfrac{7}{6} \tau_0$.  We thus obtain
\begin{align*}
    \int_M e^{-f} \brs{H_{\varphi}}^2 dV_{g_{\varphi}} =&\ \int_M e^{-f} H_{\varphi} \wedge \star H_{\varphi}\\
    =&\ \int_M H_{\varphi} \wedge \left( d (e^{-f} \varphi) + \tfrac{7}{6} \tau_0 \star \varphi \right)\\
    =&\ \int_M \left( - d H_{\varphi} \wedge e^{-f} \varphi + \tfrac{7}{6} \tau_0 \IP{H_{\varphi}, \varphi} dV_{g_{\varphi}} \right)\\
    =&\ \tfrac{49}{36} \tau_0^2 \int_M e^{-f} dV_{g_{\varphi}},
\end{align*}
as claimed.
\end{proof}
\end{prop}

\begin{rmk} A general feature of generalized Ricci solitons is that the quantity $\gD f - \brs{\N f}^2 + \tfrac{1}{6} \brs{H}^2$ is constant, and this constant is given by the Perelman quantity $\gl$ (cf. \cite{GRFbook} Proposition 4.33).  Proposition \ref{p:G2rigidity} shows that this constant is proportional to $\tau_0^2$.
\end{rmk}

\subsection{Dimension reduction}

\begin{lemma} \label{l:g2algebraic} Given $(M, \varphi)$ a $G_2$ manifold and a nonvanishing vector field $V$, then
\begin{enumerate}
    \item The following equations uniquely define an $SU(3)$ structure $(\gw, \Omega^+)$ on $V^{\perp}$:
    \begin{align*}
            \varphi =&\ \mu \wedge \omega + \Omega^+, \qquad 
    \star_{\varphi} \varphi = \tfrac{1}{2} \gw^2 - \mu \wedge \Omega^-.
    \end{align*}
    \item $g_{\varphi} = \mu \otimes \mu + g_{\gw}$.
\end{enumerate}

\begin{proof} Using the pointwise form of a $G_2$ structure in the notation of \ref{f:G2pointwise}, after identifying $\mu$ pointwise with $e_{7}$, it follows that $\varphi$ takes the indicated form pointwise in terms of $\omega_{0}$ and $\Omega^{+}_{0}$ of \ref{f:SU3Pointwise}. The relationship for the metric follows given $\varphi_{0}$ defines the seven-dimensional Euclidean metric, which is equivalent to the six-dimensional Euclidean metric defined by $(\omega_{0}, \Omega^{+})$ supplemented by the $e_{7}$ direction.
\end{proof}
\end{lemma}

We now prove the main theorem of this section, which extends (\cite{fino2025some} Theorem 6.1) (cf. also \cite{apostolov2004kahler} for $S^1$ reduction of $G_2$ holonomy metrics): 

\begin{thm} \label{t:G2reduction} Let $(M^7, \varphi)$ be a compact strong torsion $G_2$ structure.  Assume $V$ as in Theorem \ref{t:g2structure} is nonvanishing and let $(\gw, \Omega^+)$ denote the data guaranteed by Lemma \ref{l:g2algebraic}.  Then
\begin{enumerate}
    \item $(\gw, \Omega^+)$ is a skew-torsion $SU(3)$ structure.
    \item $H_{\varphi} = CS(\mu) + H_{\gw}$, where $H_{\gw} = d^c \gw + N_{\gw}$
    \item $(g_{\gw}, d \theta_{\varphi}, H_{\gw}, f)$ is a string generalized Ricci soliton.
    \item $d \theta_{\varphi} \in \Lambda^{1,1}_0$,
    \item The intrinsic torsion components of $(\gw, \Omega^+)$ satisfy
    \begin{gather*}
        \gs_0 = \tfrac{1}{2},\quad \gs_2 = 0, \quad 
        \nu_1 = \tfrac{1}{2} df,  \quad \nu_3 = \tfrac{1}{8}\tau_{0}\Omega^{-}  + \tfrac{1}{4}df\wedge\omega - i_V (\star_{\varphi} \tau_3),\\
        \pi_0 = \tfrac{7}{12} \tau_0, \quad \pi_1 = df, \quad \pi_2 = 0,\\
        e^f d (e^{-f} \Omega^-) = \tfrac{1}{2} \gw^2, \quad e^f d (e^{-f} \Omega^+) = \tfrac{7}{12} \tau_0 \gw^2.
    \end{gather*}
    \item $\theta_{\gw} = df$.
\end{enumerate}
\begin{proof} By Lemma \ref{l:reducedBC}(5) there is a connection $\hat{\N}$ with skew-symmetric torsion on the transverse distribution preserving $\gw$ and $\Omega^{\pm}$, hence item (1) follows.  From (\cite{friedrich2002parallel} Theorem 10.1), the torsion of a skew-symmetric connection is uniquely determined by (\ref{eq:AHHform}), hence items (2) and (3) follow (i.e. $\hat{H} = H_{\gw}$).  

Next, by Theorem \ref{t:g2structure}(3) we have $d \mu = d \theta_{\varphi} \in \Lambda_{14}^2$ hence
\begin{align*}
    0 = d \mu \wedge \star_{\varphi} \varphi = \tfrac{1}{2} d \mu \wedge \omega^2 - \mu \wedge d \mu \wedge \Omega^-.
\end{align*}
Taking the interior product with $V$, it follows that $d \mu \wedge \Omega^- = 0$, whence $d \mu \in \Lambda_8^2 = \Lambda^{1,1}$ (cf. Lemma 2.6 \cite{foscolo2017}).  Returning to the above equation then yields $d \mu \wedge \gw^2  = 0$, hence $\tr_{\gw} d \mu = 0$ and item (4) follows.

Turning to item (5), first note that for an $SU(3)$ structure the existence of a compatible connection with skew-symmetric torsion is equivalent to the Nijenhuis tensor being skew-symmetric, which in turn is equivalent to $\pi_2 = \gs_2 = 0$.   Next using $\star_{\varphi} \varphi = \tfrac{1}{2} \gw^2 - \mu \wedge \Omega^-$ we differentiate and obtain
\begin{align} \label{f:G2red10}
    d(\star_{\varphi} \varphi) = \gw \wedge d \gw - d \mu \wedge \Omega^- + \mu \wedge d \Omega^- = \gw \wedge d \gw + \mu \wedge d \Omega^-.
\end{align}
Since $\tau_2 = 0$ we also have (cf. \cite{fino2025some}  Proposition 5.4)
\begin{align*}
    \star_{\gw} d \Omega^- \wedge \Omega^- + 2 \gw \wedge d^c \gw - 2 d \mu \wedge \Omega^+ =&\ 0.
\end{align*}
Applying $J_{\gw}$ thus yields (cf. Lemma 2.9v \cite{foscolo2017})
\begin{align*}
    0 =&\ \star_{\gw} d \Omega^- \wedge \Omega^+ - 2 \gw \wedge d \gw + 2 d \mu \wedge \Omega^- = \star_{\gw} d \Omega^- \wedge \Omega^+ - 2 \gw \wedge d \gw.
\end{align*}
Using this in (\ref{f:G2red10}) then yields
\begin{align*}
    d(\star_{\varphi} \varphi) = \tfrac{1}{2} \star_{\gw} d \Omega^- \wedge \Omega^+ + \mu \wedge d \Omega^-.
\end{align*}
Expressing $d \Omega^- = \gs_0 \gw^2 + \pi_1 \wedge \Omega^-$ then yields
\begin{align*}
    d (\star_{\varphi} \varphi) =&\ \tfrac{1}{2} \star_{\gw} \left( \gs_0 \gw^2 + \pi_1 \wedge \Omega^- \right) \wedge \Omega^+ + \mu \wedge \left( \gs_0 \gw^2 + \pi_1 \wedge \Omega^- \right)\\
    =&\ \tfrac{1}{2} \gs_0 \gw^2 \wedge \Omega^+ + \tfrac{1}{2}  \star_{\gw} (\pi_1 \wedge \Omega^-) \wedge \Omega^+ + \mu \wedge \left( \gs_0 \gw^2 + \pi_1 \wedge \Omega^- \right)\\
    =&\ \tfrac{1}{2} \pi_1 \wedge \gw^2 + \mu \wedge \left( \gs_0 \gw^2 + \pi_1 \wedge \Omega^- \right)\\
    =&\ (2 \gs_0 \mu + \pi_1) \wedge \star_{\varphi} \varphi.
\end{align*}
On the other hand by the integrability of $\varphi$ we also have (cf. \ref{f:LeeformG2})
\begin{align*}
    d (\star_{\varphi} \varphi) = \theta_{\varphi} \wedge (\star_{\varphi} \varphi).
\end{align*}
The two equations above thus yield
\begin{align*}
    (2 \gs_0 \mu + \pi_1 - df) \wedge \star_{\varphi} \varphi = \mu \wedge (\star_{\varphi} \varphi),
\end{align*}
and comparing the horizontal and vertical components of the two sides thus yields
\begin{align*}
    \gs_0 = \tfrac{1}{2}, \quad \pi_1 = df.
\end{align*}
Turning again to (\cite{fino2025some} Proposition 5.4) and noting $\theta_{\varphi} = 4 \tau_1$, we have
\begin{align*}
    \theta_{\varphi} - df = \mu + \tfrac{2}{3} (\pi_1 + \nu_1) - df = \mu + \tfrac{1}{3} \left(2 \nu_1 - df \right).
\end{align*}
Again comparing vertical and horizontal pieces yields $\nu_1 = \tfrac{1}{2} df$.  Returning to (\ref{f:G2red10}) we obtain
\begin{align*}
    \gw \wedge d \gw + \mu \wedge d \Omega^- =&\ (\mu + df) \wedge \star_{\varphi} \varphi\\
    =&\ (\mu + df) \wedge (\tfrac{1}{2} \gw^2 - \mu \wedge \Omega^-)\\
    =&\ \mu \wedge \left( \tfrac{1}{2} \gw^2 + df \wedge \Omega^- \right) + \tfrac{1}{2} df \wedge \gw^2.
\end{align*}
The claimed equation for $d \Omega^-$ follows.

To derive the equations for $\pi_0$ and $d \Omega^+$ we first observe that equation (\ref{f:G2torsion}) implies
\begin{align*}
    \star H_{\varphi} = \tfrac{1}{6} \tau_0 \star_{\varphi} \varphi + \tfrac{1}{4} \theta_{\varphi} \wedge \varphi - \star \tau_3.
\end{align*}
On the other hand we also have
\begin{align*}
    d \varphi = \tau_0 \star_{\varphi} \varphi + \tfrac{3}{4} \theta_{\varphi} \wedge \varphi + \star_{\varphi} \tau_3.
\end{align*}
Hence we obtain
\begin{align*}
    \star_{\varphi} H_{\varphi} + d \varphi = \tfrac{7}{6} \tau_0 \star_{\varphi} \varphi + \theta_{\varphi} \wedge \varphi.
\end{align*}
Since $e^{-f} H_{\varphi}$ is co-closed by Proposition \ref{p:reducedGRS}, we thus obtain
\begin{align*}
    d (e^{-f} d \varphi) =&\ d \left( e^{-f} \left( \star_{\varphi} H_{\varphi} + d \varphi \right) \right)\\
    =&\ d \left( e^{-f} \left( \tfrac{7}{6} \tau_0 \star_{\varphi} \varphi + \theta_{\varphi} \wedge \varphi \right) \right)\\
    =&\ e^{-f} \left( - df \wedge \left( \tfrac{7}{6} \tau_0 \star_{\varphi} \varphi + \theta_{\varphi} \wedge \varphi \right) + \tfrac{7}{6} \tau_0 d (\star_{\varphi} \varphi) + d \theta_{\varphi} \wedge \varphi - \theta_{\varphi} \wedge d \varphi \right).
\end{align*}
First note that the left hand side above can be expressed as $- e^{-f} df \wedge d \varphi$, and using this yields
\begin{align*}
    0 =&\ d f \wedge d (\mu \wedge \gw + \Omega^+) - df \wedge \tfrac{7}{6} \tau_0 ( \tfrac{1}{2} \gw^2 - \mu \wedge \Omega^-) - df \wedge \theta_{\varphi} \wedge \varphi\\
    &\ + \tfrac{7}{6} \tau_0 \theta_{\varphi} \wedge (\tfrac{1}{2} \gw^2 - \mu \wedge \Omega^-) + d \theta_{\varphi} \wedge (\mu \wedge \gw + \Omega^+) - \theta_{\varphi} \wedge d (\mu \wedge \omega + \Omega^+).
\end{align*}
The first, fifth, and sixth terms above combine to yield
\begin{align*}
    d \theta_{\varphi} \wedge (\mu \wedge \gw + \Omega^+) - \mu \wedge d(\mu \wedge \gw + \Omega^+) = - \mu \wedge d \Omega^+.
\end{align*}
The second and fourth terms above simplify to
\begin{align*}
    \tfrac{7}{6} \tau_0 \mu \wedge (\tfrac{1}{2} \gw^2 - \mu \wedge \Omega^-) = \tfrac{7}{12} \tau_0 \mu \wedge \gw^2.
\end{align*}
Finally we simplify the third term
\begin{align*}
    - df \wedge \theta_{\varphi} \wedge \varphi = \mu \wedge df \wedge (\mu \wedge \gw + \Omega^+) = \mu \wedge (df \wedge \Omega^+).
\end{align*}
Combining the above equations yields
\begin{align*}
    0 = \mu \wedge \left( \tfrac{7}{12} \tau_0 \gw^2 - e^{f} d (e^{-f} \Omega^+) \right),
\end{align*}
yielding the claimed equation for $d \Omega^+$, whence $\pi_0$ as well.

To compute $\nu_{3}$, we will compute $d \varphi$ two ways.  First using the Chiossi-Salamon decomposition \ref{f:SU(3)reps}, Lemma \ref{l:g2algebraic}, and our prior computations, we have
\begin{align*}
d\varphi =&\ d\mu \wedge \omega - \mu \wedge d\omega + d\Omega^{+}\\
=&\ d\mu \wedge \omega - \mu \wedge (-\tfrac{3}{4} \Omega^{+} + \tfrac{7}{8}\tau_{0}\Omega^{-} + \tfrac{1}{2} df\wedge\omega + \nu_{3}) + \tfrac{7}{12}\tau_{0}\omega\wedge\omega + df\wedge\Omega^{+}.
\end{align*}
On the other hand we also have, again using Lemma \ref{l:g2algebraic},
\begin{align*}
    d \varphi =&\ \tau_0 \star_{\varphi} \varphi + \tfrac{3}{4} \theta_{\varphi} \wedge \varphi + \star_{\varphi} \tau_3\\
    =&\ \tau_{0}(\tfrac{1}{2}\omega\wedge\omega - \mu \wedge \Omega^{-}) + \tfrac{3}{4} \theta_{\varphi} \wedge (\mu \wedge \omega + \Omega^{+}) + \star_{\varphi} \tau_3\\
    =&\ \tau_{0}(\tfrac{1}{2}\omega\wedge\omega - \mu \wedge \Omega^{-}) + \tfrac{3}{4} (\mu + df) \wedge (\mu \wedge \omega + \Omega^{+}) + \star_{\varphi} \tau_3\\
    =&\ \tau_{0}(\tfrac{1}{2}\omega\wedge\omega - \mu \wedge \Omega^{-}) + \tfrac{3}{4} \mu \wedge ( \Omega^+ - df \wedge \omega) + \tfrac{3}{4} df \wedge \Omega^+ + \star_{\varphi} \tau_3.
\end{align*}
Comparing these two expressions and taking the interior product with $V$ yields the claimed formula.

Finally, item (6) follows from the general fact that $\theta_{\gw} = 2 \nu_1$.
\end{proof}
\end{thm}

Next we show that the symmetry reduction of Theorem \ref{t:G2reduction} is a local equivalence, building on the discussion of (\cite{fino2025some} \S 6).  In fact, we show that the Bianchi identity essentially encodes the entire structure.

\begin{thm} \label{t:G2reductionconverse} Let $(M, \gw, \Omega^+)$ be a skew-torsion $SU(3)$ structure such that $\gs_0 = \tfrac{1}{2}$, $\pi_0$ is constant, and $\theta_{\gw} = df$.  Further assume $M$ is endowed with $F \in \Lambda^2(M)$, $d F = 0, [F] \in H^2(M, \mathbb Z)$.  Finally assume one of the following equivalent conditions:
\begin{enumerate}
    \item $d H_{\gw} + F \wedge F = 0$.
    \item $\gD_{\gw} \gw = (1 + 4 \pi_{0}^2)\omega + \star_{\omega}((2\pi_{0}df + J(df)) \wedge \Omega^{+}) - \tfrac{1}{2} d\star_{\omega}(df\wedge\omega\wedge\omega) + d^{\star}(df\wedge\omega) + \star_{\omega}dJ(df\wedge\omega) + \star_{\omega}(F\wedge F)$.
\end{enumerate}
Let $\pi: \bar{M} \to M$ denote the $U(1)$ bundle associated to $[F]$, endowed with a principal connection $\mu$ such that $d \mu = F$.  Then $\varphi = \mu \wedge \gw + \Omega^+$ defines a strong torsion $G_2$ structure on $\bar{M}$.
\begin{proof} First, assuming that items (1) and (2) are equivalent, we follow the discussion of Remark \ref{r:reversibility}.  In particular, note the metric $g_{\varphi}$ is in the ansatz of Lemma \ref{l:reducedBC}(2).  Moreover, if we define $H = CS(\mu) + H_{\gw}$, it follows that the skew-torsion connection defined by $H$ preserves $g_{\varphi}$ and $\varphi$, hence $\varphi$ is a skew-torsion $G_2 $ structure, and $H = H_{\varphi}$.  Item (1) or equivalently (2) then shows that $d H_{\varphi} = 0$, hence $\varphi$ is a strong torsion $G_2$ structure.

To show the equivalence of (1) and (2) we first observe the following relationship between $J(d\omega)$ and $\star_{\omega}d\omega$ which follows from the Chiossi-Salamon decomposition, the characterizations of the $SU(3)$ representations on three-forms \cite{foscolo2017}, and the fact that $\theta_{\gw} = df$,
\begin{equation}
J(d\omega) = \star_{\omega}d\omega + J(df\wedge\omega) - \star_{\omega}(df\wedge\omega).
\end{equation}
Using this and (\ref{d:strongSU(3)T}) we obtain
\begin{equation}
H_{\omega} = -\star_{\omega}d\omega - J(df\wedge\omega) + \star_{\omega}(df\wedge\omega) + N.
\end{equation}
Hence we compute, using (\cite{fino2025some} Proposition 4.3) and that $\pi_0$ and $\gs_0$ are constant,
\begin{align*}
dH_{\omega} =&\ -d\star_{\omega}d\omega - dJ(df\wedge\omega) + d\star_{\omega}(df\wedge\omega) + dN\\
=&\ -d\star_{\omega}d\omega - dJ(df\wedge\omega) + d\star_{\omega}(df\wedge\omega) -2\Big( \tfrac{1}{4} + \pi_{0}^2\Big)\omega\wedge\omega -2 \Big(\pi_{0}df + \tfrac{1}{2}J(df)\Big)\wedge\Omega^{+}.
\end{align*}
Also we may compute, using the Chiossi-Salamon decomposition and again that $\theta_{\gw} = df$, 
\begin{align*}
d\star_{\omega}d\star_{\omega}\omega =\tfrac{1}{2} d\star_{\omega}d(\omega\wedge\omega) = d\star_{\omega}(\omega\wedge d\omega) = \tfrac{1}{2}d\star_{\omega}(df\wedge\omega\wedge\omega),
\end{align*}
which implies 
\begin{align*}
    0 =&\ -d d^* \gw - \tfrac{1}{2} d \star_{\gw} (df \wedge \gw^2).
\end{align*}
Combining this with the above we yield
\begin{align*}
dH_{\omega} =&\ -d\star_{\omega}d\omega - dJ(df\wedge\omega) + d\star_{\omega}(df\wedge\omega) -2\Big( \tfrac{1}{4} + \pi_{0}^2\Big)\omega\wedge\omega -2 \Big(\pi_{0}df + \tfrac{1}{2}J(df)\Big)\wedge\Omega^{+}\\
=&\ \star_{\gw} \left[ \gD_{\gw} \gw - \tfrac{1}{2} d \star_{\gw} (df \wedge \omega^2) \right.\\
&\ \quad \left. - \star_{\gw} \left( d J (df \wedge \gw) + d\star_{\omega}(df\wedge\omega) -2\Big( \tfrac{1}{4} + \pi_{0}^2\Big)\omega\wedge\omega -2 \Big(\pi_{0}df + \tfrac{1}{2}J(df)\Big)\wedge\Omega^{+} \right) \right]. 
\end{align*}
The equivalence of (1) and (2) follows directly from this.
\end{proof}
\end{thm}

\begin{cor} (cf. \cite{fino2023twisted} pg 23) \label{c:G2splitting} In the setup of Theorem \ref{t:G2reduction}, the following are equivalent:
\begin{enumerate}
    \item $d H_{\gw} = 0$,
    \item $d \mu = d \theta_{\varphi} = 0$,
    \item $D \mu = 0$, i.e. $M$ locally splits as a Riemannian product.
\end{enumerate}
\begin{proof} For an $SU(3)$ structure $(\gw, \Omega^+)$, a differential form $\ga \in \Lambda^{1,1}_0$ satisfies
\begin{align*}
    \star_{\gw} (\ga \wedge \ga) = - \tfrac{1}{3} \brs{\ga}^2 \gw + \gb,
\end{align*}
where $\gb \in \Lambda^{1,1}_0$.  Applying this to $d \mu = d \theta_{\varphi}$ and using that $d H_{\gw} = d \mu \wedge d \mu$ the equivalence of items (1) and (2) easily follows.  On the other hand since $V$ is Killing we also have
\begin{align*}
    D \mu = L_{V} g + d \mu = d \theta_{\varphi},
\end{align*}
and the equivalence of (2) and (3) follows.
\end{proof}
\end{cor}

\begin{ex} \label{ex:nonintG2} (\cite{fino2023twisted}, cf. \cite{ivanov2023riemannianG2} Ex. 7.6)
Let $G = SU(2) \times SU(2) \times U(1)$ with the following structure equations
\begin{align*}
de_{1} &= e_{23}, \hspace{10mm} de_{2}=e_{31}, \hspace{10mm} de_{3}=e_{12}, \\
de_{4} &= e_{56}, \hspace{10mm} de_{5}=e_{64}, \hspace{10mm} de_{6} = e_{45}, \hspace{10mm} de_{7} = 0.
\end{align*}
Define a $G_2$ structure $\varphi$ in this frame of the form
\begin{equation}
\varphi = e_{147} + e_{257} - e_{367} + e_{123} + e_{156} - e_{246} - e_{345}\nonumber
\end{equation}
This structure is skew-torsion and a computation yields that the Lee form is $\theta_{\varphi} = e_{7}$, which is closed.  Using Lemma \ref{l:g2algebraic}, the reduced $SU(3)$-structure is
\begin{align*}
\omega &= e_{14} - e_{25} - e_{36}, \quad 
\Omega^{+} = e_{123} + e_{156} - e_{246} - e_{345}.\nonumber
\end{align*}
The $G_2$ structure $\varphi$ is of constant type, in particular $d\varphi\wedge\varphi=7\textnormal{vol}$. Using this fact and that $df=0$, it follows from Lemma \ref{l:g2algebraic} that the non-vanishing torsion forms of the $SU(3)$ structure are $\sigma_{0}$, $\pi_{0}$, and $\nu_{3}$. Following a complex rotation of $\Omega^{\pm}$ (cf. \cite{fino2025some} \S 6) this structure yields a co-coupled half-flat strong torsion $SU(3)$ structure. Observe that this example is distinct from that in Example \ref{ex:nonintsu(3)}, indeed that example cannot come from our dimension reduction as $\gs_0$ takes the value $-2$ versus $\tfrac{1}{2}$.
\end{ex}

\begin{ex} \label{ex:nonintG2nonclosedLee} (\cite{fino2023twisted}, cf. \cite{ivanov2023riemannianG2} Ex. 7.7)
Let $G = SU(2) \times SU(2) \times U(1)$ with the same structure equations as Example \ref{ex:nonintG2}.  Consider the $G_2$ three-form defined by (\ref{f:G2pointwise}) in this frame.  A computation yields that the $G_2$ structure is strong torsion with Lee form $\theta_{\varphi} = e_{4} - e_{3}$.  In particular, we have 
\begin{equation}
d\theta_{\varphi} = e_{56} - e_{12}\nonumber.
\end{equation}
Furthermore, by Lemma \ref{l:g2algebraic} we obtain the reduced $SU(3)$ structure
\begin{align*}
\omega &= e_{16} + e_{25} - e_{37} + e_{15} - e_{26} - e_{47}, \quad 
\Omega^{+} = e_{127} - e_{146} - e_{245} + e_{567} - e_{136} - e_{235},\nonumber
\end{align*}
which is well-defined on the quotient $G / \IP{\theta_{\varphi}^{\sharp}} \cong S^3 \times \mathbb {CP}^1 \times S^1$.  
In particular it is a skew-torsion $SU(3)$ structure This defines a string generalized Ricci soliton with nontrivial $F$ which becomes co-coupled half-flat after a complex rotation of $\Omega^{\pm}$ (cf. \cite{fino2025some}).
\end{ex}

\section{Strong torsion \texorpdfstring{$\Spin(7)$}{Spin(7)} structures} \label{s:Spin7}

In this section we begin by recalling fundamental definitions regarding $\Spin(7)$ structures, their torsion, the fact that they always admit a connection with skew-symmetric torsion, and the recent result of Ivanov-Petkov \cite{ivanov2023riemannianspin7} showing that those with strong torsion are automatically generalized Ricci solitons and admit a canonical symmetry (Theorem \ref{t:spin(7)structure}).  Next in Theorem \ref{t:Spin(7)reduction} we show the transverse $G_2$ structure is skew-torsion, and a string generalized Ricci soliton.  Moreover we compute its torsion, showing that it is conformally co-closed, and balanced, with conformal factor $e^{-f}$.  We end by showing a rigidity phenomena.

\subsection{Definitions and structure theorems}

\begin{defn}\label{Spin(7)-Structure} The data $(M^8,\Psi)$ defines a \emph{Spin(7) structure} \cite{bonan1966varietes} if $\Psi\in\Lambda^{4}$ may be expressed pointwise with $\Psi_{0}$ in the following manner. \begin{gather} \label{f:spin74form} 
\begin{split}
    \Psi_{0} =&\ -e_{1238} + e_{1347} - e_{1458} - e_{1678} + e_{1257} + e_{1356} - e_{1246}\\
    &\ - e_{4567} - e_{2568} - e_{2367} - e_{2345} - e_{3468} - e_{2478} + e_{3578}.
\end{split}
\end{gather}
The subgroup of $GL(8,\mathbb R)$ preserving $\Psi$ is $\Spin(7)$.  As this group is a subgroup of $SO(8)$ the four-form $\Psi$ induces a Riemannian metric $g_{\Psi}$, an explicit formula for which appears in (\cite{karigiannis2005deformations} Theorem 4.3.5).
\end{defn}

On a manifold with $\Spin(7)$ structure, we obtain decompositions of differential forms into irreducible representations of $\Spin(7)$ as:
\begin{align*}
\Lambda^{1} &= \Lambda_{8}^{1}, \quad \Lambda^{2} = \Lambda_{7}^{2} \oplus \Lambda_{21}^{2},\\
\Lambda^{3} &= \Lambda_{8}^{3}\oplus \Lambda_{48}^{3}, \quad 
\Lambda^{4} = \Lambda_{1}^{4} \oplus \Lambda_{7}^{4} \oplus \Lambda_{27}^{4} \oplus \Lambda_{35}^{4}.
\end{align*}
These representation spaces are explicitly described via (cf. \cite{joyce2000compact}):
\begin{align*}
\Lambda_{7}^{2} &= \{\beta\in\Lambda^{2} | \star(\Psi\wedge\beta) = -3\beta\}\\
\Lambda_{21}^{2} &= \{\beta\in\Lambda^{2}|\star(\Psi\wedge\beta) = \beta\}\\
\Lambda_{8}^{3} &= \{ X\lrcorner\Psi | X\in\Gamma(TM)\}\\
\Lambda_{48}^{3} &= \{\gamma\in\Lambda^{3}|\gamma\wedge\Psi=0 \}.
\end{align*}

\begin{defn} Given a $\Spin(7)$ structure $\Psi$, from \cite{fernandez1986classification} we define torsion forms $\theta_{\Psi}$ and $\zeta_5$ via
\begin{gather*}
\begin{split}
d \Psi = \theta_{\Psi} \wedge \Psi + \zeta_5,
\end{split}
\end{gather*}
where 
\begin{equation} \label{f:spin(7)LeeForm}
\theta_{\Psi} = -\tfrac{1}{7} \star_{\Psi} (\star_{\Psi}d\Psi \wedge \Psi)    
\end{equation}
is the \emph{Lee form} of $\Psi$.  A fundamental result of Ivanov (\cite{ivanov2002geometry} Theorem 1, \cite{ivanov2004connections} Theorem 1.1) yields that $\Spin(7)$ manifolds \emph{always} admit a unique compatible connection with skew-symmetric torsion $H_{\Psi}$ given by 
\begin{align*}
    H_{\Psi} = - \star_{\Psi} d \Psi + \tfrac{7}{6} \star_{\Psi} (\theta_{\Psi} \wedge \Psi).
\end{align*} 
We say that the structure has \emph{strong torsion} if $d H_{\Psi} = 0$.
\end{defn}

\begin{prop}\label{p:strongSpin(7)Properties}
Let $(M, \Psi)$ be a strong torsion $\Spin(7)$ structure. Then
\begin{align} \label{f:spin7dilatino}
    0 =&\ \tfrac{7}{6}d^{\star}\theta_{\Psi} + \tfrac{7}{6}|\theta_{\Psi}|^2 - |\zeta_{5}|^2.
\end{align}
In particular, 
\begin{enumerate}
\item $\theta_{\Psi} = 0$ implies $\Psi$ is torsion-free.
\item Assuming $M$ is compact, $\zeta_{5}=0$ implies $\Psi$ is torsion-free.
\end{enumerate}
\end{prop}

\begin{proof} The vanishing of the $\Lambda_{1}^{4}$ dimensional component of $dH_{\Psi}$ is equivalent to
\begin{align*}
0 =&\ dH_{\Psi}\wedge\Psi\\
=&\ d (\tfrac{1}{6}\theta_{\Psi}\lrcorner\Psi - \star_{\Psi}\zeta_{5})\wedge\Psi\\
=&\ d(\tfrac{1}{6}\theta_{\Psi}\lrcorner\Psi\wedge\Psi) + (\tfrac{1}{6}\theta_{\Psi}\lrcorner\Psi - \star_{\Psi}\zeta_{5})\wedge(\theta_{\Psi}\wedge\Psi + \star_{\Psi}\zeta_{5})\\
=&\ -\tfrac{7}{6}d\star_{\Psi}\theta_{\Psi} + \tfrac{7}{6} |\theta_{\Psi}|^2 dV_{g_{\Psi}} - |\zeta_{5}|^2 dV_{g_{\Psi}}\\
=&\ \left( \tfrac{7}{6}d^{\star}\theta_{\Psi} + \tfrac{7}{6}|\theta_{\Psi}|^2 - |\zeta_{5}|^2 \right) dV_{g_{\Psi}}.
\end{align*}
This yields equation (\ref{f:spin7dilatino}).  Item (1) then follows easily, and item (2) follows by integrating.
\end{proof}

\begin{thm} \label{t:spin(7)structure} (cf. \cite{ivanov2023riemannianspin7} Theorems 6.1, 6.4) Let $(M, \Psi)$ be a compact strong torsion $\Spin(7)$ manifold.  Then
\begin{enumerate}
    \item $\Rc^{\N} + \tfrac{7}{6} \N \theta_{\Psi} = 0$,
    \item $d \theta_{\Psi} \in \Lambda_{21}^2$,
    \item There exists a unique smooth function $f$ such that $\int_M e^{-f} dV_g = 1$ and $\Rc^{\N} + \N df = 0$,
    \item The vector field $V = \tfrac{7}{6} \theta^{\sharp}_{\Psi} - \N f$ satisfies $L_V \Psi = 0$,
    \item $V$ vanishes at one point if and only if $V = 0$, if and only if $H_{\Psi} = 0$ and $f$ is constant.
\end{enumerate}
\end{thm}

\begin{proof} Items (1) and (2) follow from (\cite{ivanov2023riemannianspin7} Theorem 6.1).  Items (3) and (4) follow from (\cite{ivanov2023riemannianspin7} Theorem 6.4), which is stated as an equivalence although case (a) of that theorem holds a priori.  Also item (3) is explained in Proposition \ref{p:reducedGRS} above.  Item (5) is explained in \cite{ivanov2023riemannianspin7} Theorem 6.4.
\end{proof}

\begin{rmk} As our interest is in capturing solitons with nontrivial torsion, due to item (5) above we will assume throughout below that the canonical vector field $V$ is nonvanishing.
\end{rmk}

\subsection{Dimension Reduction}

\begin{lemma} \label{l:spin(7)algebraic} Given $(M, \Psi)$ a $\Spin(7)$ structure and a unit length vector field $V$, then
\begin{enumerate}
    \item $\varphi = i_V \Psi$ is a $G_2$ structure on $\spn \{V\}^{\perp}$,
    \item $\Psi = \mu \wedge \varphi + \star_{\varphi} \varphi$,
    \item $g_{\Psi} = \mu \otimes \mu + g_{\varphi}$.
\end{enumerate}

\begin{proof} This follows from the discussion for instance in (\cite{foscolo2021complete} \S 3.2.1), noting we are in the special case when $V$ has unit length.
\end{proof}
\end{lemma}

\begin{thm} \label{t:Spin(7)reduction} Let $(M, \Psi)$ be a compact strong torsion $\Spin(7)$ structure.  Fix $V$ as in Theorem \ref{t:spin(7)structure} and let $\varphi = i_V \Psi$.  Then
\begin{enumerate}
    \item $\varphi$ is a skew-torsion $G_2$ structure,
    \item $H_{\Psi} = CS(\mu) + H_{\varphi}$, where $H_{\varphi} = - \star_{\varphi} d \varphi + \star_{\varphi} (\theta_{\varphi} \wedge \varphi) + \tfrac{1}{6} \IP{d \varphi, \star_{\varphi} \varphi} \varphi$,
    \item $(g_{\varphi}, \tfrac{7}{6} d \theta_{\Psi}, H_{\varphi},f)$ is a string generalized Ricci soliton,
    \item $d \theta_{\Psi} \in \Lambda^2_{14}$,
    \item The intrinsic torsion components of $\varphi$ satisfy
    \begin{gather*}
        \tau_0 = - \tfrac{6}{7} ,\quad \tau_2 = 0, \quad 
        e^{f} d (e^{-f} \star_{\varphi} \varphi) = 0,\\
        \star_{\varphi}\tau_{3} = \tfrac{3}{28}\theta_{\varphi}\wedge\varphi - i_{V}\zeta_{5},
    \end{gather*}
    \item $\theta_{\varphi} = df$.
\end{enumerate}
\begin{proof} 
First observe by Lemma \ref{l:reducedBC} that $H_{\Psi} = CS(\mu) + \hat{H}$ for a basic three-form $\hat{H}$.  By Lemma \ref{l:reducedBC}(3) and Lemma \ref{l:spin(7)algebraic}(3) it follows that $\hat{\N} g_{\varphi} = 0$.  Moreover since $L_{V} \Psi = 0$ by Theorem \ref{t:spin(7)structure}(4), it follows from Lemma \ref{l:reducedBC} that $\hat{\N} \varphi = 0$, hence item (1) follows.  Moreover by uniqueness of the Bismut connection, it follows that $\hat{H} = H_{\varphi}$, establishing item (2).  Item (3) follows from Proposition  \ref{p:reducedGRS}.

To establish item (4), we start with the fact that $d \theta_{\Psi} \in \Lambda^2_{21}$ by Theorem \ref{t:spin(7)structure}(2), which is equivalent to $\star_{\Psi} (\Psi \wedge d \theta_{\Psi}) = d \theta_{\Psi}$.  Now using the algebraic decomposition of $\Psi$ this implies
\begin{align*}
    d \theta_{\Psi} =&\ \star_{\Psi} \left( \left( \mu \wedge \varphi + \star_{\varphi} \varphi \right) \wedge d \theta_{\Psi} \right)\\
    =&\ - \star_{\varphi} (\varphi \wedge d \theta_{\Psi} + \mu \wedge \star_{\varphi} (\star_{\varphi} \varphi \wedge d \theta_{\Psi}).
\end{align*}
As $d \theta_{\Psi}$ is basic, taking the interior product with $V$ yields $0 = \star_{\varphi} (\star_{\varphi} \varphi \wedge d \theta_{\Psi})$, which then implies $0 = \star_{\varphi} \varphi \wedge d \theta_{\Psi}$, as claimed.

To establish item (5), first observe that we have already established that $\varphi$ is skew-torsion, which by (\cite{friedrich2002parallel} Theorem 4.8) implies that $\tau_2 = 0$.  To derive the remaining components we use item (4) together with (\cite{fowdar2020} Proposition 3.2), which in our notation implies that there exists a basic function $h$ such that
\begin{align*}
    \tau_0 = - h, \quad 
    \theta_{\Psi} = h \mu + \tfrac{6}{7} \theta_{\varphi}.
\end{align*}
The second equation easily implies
\begin{align*}
    0 = \mu \left( h - \tfrac{6}{7} \right) + \tfrac{6}{7} (\theta_{\varphi} - df).
\end{align*}
The vanishing of the vertical component gives $h = \tfrac{6}{7}$, yielding then the claimed equation for $\tau_0$.  Moreover the vanishing of the horizontal component yields item (6).  The equation for $d (e^{-f} \star_{\varphi} \varphi)$ follows directly using the vanishing of $\tau_2$ and item (6).

In order to compute $\star\tau_{3}$ for the dimensionally-reduced $G_2$ structure, we start by taking the derivative of the Spin(7) four-form decomposed into horizontal and vertical components \ref{l:spin(7)algebraic}.
\begin{align*}
d\Psi &= d\mu \wedge \varphi - \mu \wedge d\varphi + d\star_{\varphi}\varphi\\
&=d\mu \wedge \varphi - \mu\wedge (-\tfrac{6}{7}\star_{\varphi}\varphi + \tfrac{3}{4}\theta_{\varphi}\wedge\varphi + \star_{\varphi}\tau_{3}) + \theta_{\varphi}\wedge\star_{\varphi}\varphi
\end{align*}

We now equate the derivative of the Spin(7) four-form expressed in this manner with the decomposition we get from the Fernandez classification \cite{fernandez1986classification}.
\begin{align}
d\Psi &= \theta_{\Psi}\wedge\Psi + \zeta_{5} = (\tfrac{6}{7}\mu + \tfrac{6}{7}\theta_{\varphi})\wedge(\mu\wedge\varphi + \star_{\varphi}\varphi) + \zeta_{5}.
\end{align}
Isolating the the vertical terms in the above expressions allows us to solve for $\star_{\varphi}\tau_{3}$.
\end{proof}
\end{thm}

\begin{rmk} \label{r:spin7asG2instanton} In particular, the reduced structure is an integrable $G_2$ structure of constant type, and moreover (canonically) conformally co-closed.  Furthermore, due to item (2), from \cite{lotay2024coupled} we obtain a canonically associated $G_2$-instanton on $(TM / \IP{V}) \oplus \mathbb R$.
\end{rmk}

\begin{thm} \label{t:spin(7)reductionconverse} Let $(M, \varphi)$ be a constant-type skew-torsion $G_2$ structure such that $\theta_{\varphi} = df$, further endowed with $F \in \Lambda^2(M)$, $d F = 0, [F] \in H^2(M, \mathbb Z)$.  Further assume one of the following equivalent conditions:
\begin{enumerate}
    \item $d H_{\varphi} + F \wedge F = 0$,
    \item $\gD_{\varphi} \varphi + \tfrac{7}{6} \tau_0 \star_{\varphi} d \varphi - d^{\star} (df \wedge \varphi) + d (i_{\N f} \varphi) + \star_{\varphi} (F \wedge F) = 0.$
\end{enumerate}
Let $\pi: \bar{M} \to M$ denote the $U(1)$ bundle associated to $[F]$, endowed with a principal connection $\mu$ such that $d \mu = F$.  Then $\Psi = \mu \wedge \varphi + \star_{\varphi} \varphi$ defines a strong torsion $\Spin(7)$ structure on $\bar{M}$.
\begin{proof} We show that items (1) and (2) are equivalent, after which the rest of the proof is formally identical to Theorem \ref{t:G2reductionconverse}.  Starting with the expression (\ref{f:G2torsion}) for $H_{\varphi}$, we differentiate, using that $\varphi$ is constant type and $\theta_{\varphi} = df$ to obtain
\begin{equation*}
dH_{\varphi} = -d\star_{\varphi}d\varphi + d\star_{\varphi} (d f \wedge\varphi) + \tfrac{7}{6}\tau_{0}d\varphi.
\end{equation*}
Furthermore, since $\varphi$ has skew-torsion, we can also compute
\begin{align*}
dd^{\star}\varphi =&\ -d\star_{\varphi}(\theta_{\varphi}\wedge\star_{\varphi}\varphi) = -d(i_{\theta_{\varphi}^{\sharp}} \varphi) = - d( i_{\N f} \varphi).
\end{align*}
Putting these computations together yields
\begin{align*}
    \star_{\varphi} \left(d H_{\varphi} + F \wedge F \right) =&\ \star_{\varphi} \left( - d \star_{\varphi} d \varphi + d\star_{\varphi} (d f \wedge\varphi) + \tfrac{7}{6}\tau_{0}d\varphi + F \wedge F\right)\\
    =&\ d^{\star} d \varphi - d^{\star} (df \wedge \varphi) + \tfrac{7}{6} \tau_0 \star_{\varphi} d \varphi + \star_{\varphi} (F \wedge F)\\
    =&\ d^{\star} d \varphi + d d^{\star} \varphi + \tfrac{7}{6} \tau_0 \star_{\varphi} d \varphi - d^{\star} (df \wedge \varphi) + d (i_{\N f} \varphi) + \star_{\varphi} (F \wedge F)\\
    =&\ \gD_{\varphi} \varphi + \tfrac{7}{6} \tau_0 \star_{\varphi} d \varphi - d^{\star} (df \wedge \varphi) + d (i_{\N f} \varphi) + \star_{\varphi} (F \wedge F).
\end{align*}
The equivalence of (1) and (2) follows.
\end{proof}
\end{thm}

\begin{cor} \label{c:spin(7)splitting} In the setup of Theorem \ref{t:Spin(7)reduction}, the following are equivalent:
\begin{enumerate}
    \item $d H_{\varphi} = 0$,
    \item $d \mu = d \theta_{\Psi} = 0$,
    \item $D \mu = 0$, i.e. $M$ locally splits as a Riemannian product.
\end{enumerate}
\begin{proof} For a $G_2$ structure $\varphi$, a differential form $\beta \in \Lambda^{2}_{14}$ satisfies \cite{bryant2003some}
\begin{align*}
    \star_{\varphi} (\beta \wedge \beta) = - \tfrac{1}{7} \brs{\beta}^2 \varphi + \gamma,
\end{align*}
where $\gamma \in \Lambda^{3}_{27}$.  Applying this to $d \mu = d \theta_{\Psi}$ and using that $d H_{\varphi} = d \mu \wedge d \mu$ the equivalence of items (1) and (2) easily follows.  On the other hand since $V$ is Killing we also have
\begin{align*}
    D \mu = L_{V} g + d \mu = d \theta_{\varphi},
\end{align*}
and the equivalence of (2) and (3) follows.
\end{proof}
\end{cor}

\begin{ex} \label{ex:nonintSpin7OneA} (cf. \cite{ivanov2023riemannianspin7} Ex. 6.5a)
Consider the group $G = U(1)\times SU(2) \times SU(2) \times U(1)$ with the structure equations
\begin{align*}
& de_{0} = 0 \hspace{5mm} de_{1} = e_{23} \hspace{5mm} de_{2} = e_{31} \hspace{5mm} de_{3} = e_{12},\\
& de_{4} = e_{56} \hspace{5mm} de_{5} = e_{64} \hspace{5mm} de_{6} = e_{45} \hspace{5mm} de_{7} = 0.\nonumber
\end{align*}
We obtain a $\Spin(7)$ structure via equation (\ref{f:spin74form}) in this frame.  The Lee form is computed to be $\theta_{\Psi} = \tfrac{6}{7}(e_{4}-e_{3})$ and this Lee form is not closed.  The vector field $\theta_{\Psi}^{\sharp}$ generates a $U(1)$ action with quotient $G / U(1) \cong \mathbb CP^1 \times S^3 \times T^2$.  Using Lemma \ref{l:spin(7)algebraic} we recover the $S^1$-reduced $G_2$ structure
\begin{align*}
\varphi =&\ \tfrac{6}{7} ( -e_{026} - e_{037} - e_{047} + e_{016} + e_{025} - e_{015}\\ 
&\quad  + e_{356} + e_{456} - e_{157} + e_{124} + e_{123} - e_{257} - e_{167} - e_{267}),
\end{align*}
which by Theorem \ref{t:Spin(7)reduction} is a co-closed $G_2$ structure such that $(g_{\varphi}, \tfrac{7}{6} d \theta_{\Psi}, H_{\varphi})$ defines a string generalized Ricci soliton.  Note that the entire structure is furthermore invariant under the obvious $U(1)^2$ action.
\end{ex}


\begin{ex} \label{ex:nonintSpin7Two} (cf. \cite{ivanov2023riemannianspin7} Ex. 6.6)
Consider $G = SU(3)$ with structure equations
\begin{align*}
& de_{1} = de_{2} = de_{4} = 0 \hspace{5mm} de_{3} = -e_{12} \hspace{5mm} de_{5} = -\tfrac{1}{2}e_{34} \hspace{5mm} de_{6} = \tfrac{1}{2}e_{15} - \tfrac{1}{2}e_{24} \\
& de_{7} = -\tfrac{1}{2}e_{14} - \tfrac{1}{2}e_{25} + \tfrac{1}{2}e_{36} \hspace{5mm} de_{0} = -\tfrac{\sqrt{3}}{2}\Big(e_{45} + e_{67} \Big).
\end{align*}
We define a Spin(7) structure $\Psi$ via equation (\ref{f:spin74form}) in this frame.  As the flat Cartan connection associated to the left action preserves $\Psi$ it is the canonical connection with skew-symmetric torsion associated to $\Psi$.  Moreover as the Cartan three-form is closed, $\Psi$ is a strong torsion structure.  We may compute the Lee form from (\ref{f:spin(7)LeeForm})
\begin{align*}
\theta_{\Psi} = \tfrac{1}{7}( (\sqrt{3} + 1)e_{2} -e_{3} - 2e_{4} - (\sqrt{3}-1)e_{5} - e_{6} + e_{7}).
\end{align*}
As such the Lee form $\theta_{\Psi}$ is not closed
\begin{equation*}
d\theta_{\Psi} = \tfrac{1}{14}(2e_{12} + (\sqrt{3}-1)e_{34} - e_{15} + e_{24} -e_{14} -e_{25}+ e_{36}).
\end{equation*}
Using Lemma \ref{l:spin(7)algebraic} we get the reduced $G_2$ three-form
\begin{align*}
\varphi =&\ \tfrac{1}{7} (e_{012} + e_{034} + e_{056} + e_{145} + e_{235} + e_{136} - e_{246} - e_{026}  \\
&\quad - e_{047} + e_{015} + e_{456} + e_{124} - e_{257} - e_{167} + e_{023} + e_{057} \\ 
&\quad + e_{014} - e_{345} - e_{125} + e_{137} - e_{247} ) \nonumber\\
&\quad + \tfrac{\sqrt{3}+1}{7}(-e_{017} - e_{036} - e_{045} + e_{126} + e_{134} - e_{357} + e_{467})\nonumber\\
&\quad + \tfrac{\sqrt{3}-1}{7}(-e_{067} + e_{024} - e_{013} + e_{346} + e_{147} + e_{126} + e_{237})\nonumber\\
&\quad + \tfrac{2}{7}(e_{037} - e_{016} - e_{025} - e_{356} - e_{157} - e_{123} + e_{267}).
\end{align*}
From inspection of the coefficients of $\theta_{\Psi}$ it follows that $\theta_{\Psi}^{\sharp}$ generates an $\mathbb R$-action, and hence the $G_2$ form does not exist on any of the Aloff-Wallach spaces \cite{aloff1975infinite}.
\end{ex}


\end{document}